\documentclass[11pt,reqno]{amsart}
\usepackage[hidelinks]{hyperref}
\usepackage{caption}
 \usepackage{amssymb,amsfonts,amscd,amsbsy, color, amsmath}
\usepackage[mathscr]{eucal}
\usepackage{url}
\usepackage{graphicx}
\usepackage{mathtools}
\usepackage{todonotes}
\usepackage{tabularx}
\usepackage{float, color}
\usepackage{placeins}
\usepackage{tikz}
\usepackage{cite}
\usepackage{bm}
\usepackage{enumitem}

\usepackage{verbatim}
\usepackage{soul}
\usepackage{graphicx}

\newtheorem{theorem}{Theorem}[section]
\newtheorem{lemma}[theorem]{Lemma}
\newtheorem{proposition}[theorem]{Proposition}
\newtheorem{corollary}[theorem]{Corollary}

\newtheorem{conjecture}[theorem]{Conjecture}
\theoremstyle{definition}
\newtheorem{remark}[theorem]{Remark}

\numberwithin{equation}{section}

\newcommand{\bea}{\begin{equation}\begin{aligned}}
\newcommand{\eea}{\end{aligned}\end{equation}}
\renewcommand{\=}{\;=\;}

\def\qqquad{\qquad\quad}

\newcommand{\Z}{\mathbb{Z}}

\makeatletter
\def\imod#1{\allowbreak\mkern5mu({\operator@font mod}\,\,#1)}
\makeatother
\allowdisplaybreaks

\makeatletter
\@namedef{subjclassname@2020}{%
  \textup{2020} Mathematics Subject Classification}
\makeatother

\begin{document}

\title{Bailey pairs, Eichler integrals and unified Witten-Reshetikhin-Turaev invariants}

\author{Jeremy Lovejoy}
\author{Robert Osburn}
\author{Matthias Storzer}

\address{CNRS, Universit{\'e} Paris Cit{\'e}, B{\^a}timent Sophie Germain, Case Courier 7014, 8 Place Aur{\'e}lie Nemours, 75205 Paris Cedex 13, France}

\email{lovejoy@math.cnrs.fr}

\address{School of Mathematical Sciences, University College Cork, Cork, Ireland}

\email{robert.osburn@ucc.ie}
\email{mstorzer@ucc.ie}

\subjclass[2020]{11F27, 33D15, 57K16}
\keywords{Bailey pairs, $q$-multisums, Eichler integrals}

\date{\today}

\begin{abstract}
In 1999, Lawrence and Zagier expressed the Witten-Reshetikhin-Turaev (WRT) invariant of the Poincar{\'e} homology sphere as the limiting value of the Eichler integral of a weight 3/2 modular form.  Habiro's construction of the unified WRT invariant subsequently recast this result as an identity for a $q$-hypergeometric series at roots of unity.  This motivated Hikami to prove analogous $q$-series identities involving the unified WRT invariants of certain Brieskorn homology spheres.  Hikami also made several conjectures of a similar type for $q$-series with no apparent connection to quantum invariants.  In this paper we use the Bailey pair machinery and a novel relation between incomplete quadratic Gauss sums with periodic coefficients to construct infinite families of identities between $q$-multisums at roots of unity and limiting values of Eichler integrals of weight 3/2 modular forms.  These identities include all of Hikami's results and conjectures as well as a generalization of the result of Lawrence and Zagier. 
\end{abstract}

\maketitle

\section{Introduction}
In 1999, Lawrence and Zagier made a prescient observation connecting quantum invariants and modular forms \cite{lz}.   To state their result, let $Z_{\zeta}(\mathcal{M})$ denote the Witten-Reshetikhin-Turaev (WRT) invariant of a closed and oriented 3-manifold $\mathcal{M}$ at the root of unity $\zeta$. Next let $\bm{p} = (p_1, p_2, p_3)$ be a triple of pairwise coprime positive integers and $P:=p_1 p_2 p_3$. For any triple $\bm{\ell} = (\ell_1, \ell_2, \ell_3) \in \mathbb{Z}^3$ satisfying $0 < \ell_j < p_j$, define the odd periodic function 
\begin{eqnarray*}
\chi_{\bm{p}}^{\bm{\ell}}(n) = 
\begin{cases}
-\epsilon_1 \epsilon_2 \epsilon_3 & \text{if $n \equiv P \left( 1 + \displaystyle \sum_{j=1}^{3} \frac{\epsilon_j \ell_j}{p_j} \right) \pmod{2P}$,} \\
0 & \text{otherwise}
\end{cases}
\end{eqnarray*}
where $\bm{\epsilon} = (\epsilon_1, \epsilon_2, \epsilon_3) \in \{ \pm 1 \}^3$. For $q=e^{2 \pi i \tau}$, $\tau \in \mathbb{H}$, consider the Eichler integral
\begin{equation} \label{ei}
\widetilde{\Phi}_{\bm{p}}^{\bm{\ell}}(\tau) := \sum_{n=0}^{\infty} \chi_{\bm{p}}^{\bm{\ell}}(n) q^{\frac{n^2}{4P}}
\end{equation}
of the weight $3/2$ modular form
\begin{equation*}
\Phi_{\bm{p}}^{\bm{\ell}}(\tau) := \sum_{n=0}^{\infty} n \chi_{\bm{p}}^{\bm{\ell}}(n) q^{\frac{n^2}{4P}}.
\end{equation*} 
Their main result \cite[Theorem 1]{lz} reads as follows.   For coprime positive integers $M$ and $N$, let $\zeta_N^M := e^{\frac{2 \pi i M}{N}}$.   Then for the Poincar\'e homology sphere $\Sigma(2,3,5)$ we have   
\begin{equation} \label{LZ0}
1 + \zeta_N^M(1-\zeta_N^M) Z_{\zeta_N^M}(\Sigma(2,3,5)) =  -\frac{1}{2} \lim_{\tau \to \frac{M}{N}} q^{-\frac{1}{120}} \widetilde{\Phi}_{(2,3,5)}^{(1,1,1)} (\tau).
\end{equation}
They also established that the right-hand side of (\ref{LZ0}) has ``a modular property modulo smooth functions", a notion which would later become the defining attribute of Zagier's quantum modular forms \cite{zquantum}.  Any Eichler integral defined by (\ref{ei}) is now known to be an example of a quantum modular form \cite[Section 4.4]{go}. 

Following up on the work of Lawrence and Zagier, Hikami wrote a series of papers \cite{hikami1, hikami2, hikami3, hikami4} wherein he showed that limiting values of Eichler integrals give the WRT invariants for several families of $3$-manifolds, including the Brieskorn homology spheres.  As an application, he used the modularity to find asymptotic expansions for $Z_{\zeta_N}(\mathcal{M})$ as $N \to \infty$.  Hikami's results have recently been used in the study of limiting values of an SU(2) WRT function \cite{am, ahlmss, fimt} and towards progress on a quantum modularity conjecture for WRT invariants \cite[Theorem 1]{ps}. 

The identity \eqref{LZ0} may be recast as an identity for $q$-hypergeometric series at roots of unity using Habiro's unified WRT invariant.  Habiro \cite{Habiro} showed that if $\mathcal{M}$ is an integral homology sphere, there is a $q$-hypergeometric series $I(\mathcal{M};q)$ such that for any root of unity $\zeta$ one has
\begin{equation*}
I(\mathcal{M};\zeta) = Z_{\zeta}(\mathcal{M}).
\end{equation*}   
In the case of the Poincar\'e homology sphere, the unified WRT invariant is \cite{TTQLe}
\begin{equation*}
I(\Sigma(2,3,5);q) = \frac{1}{1-q}\sum_{n \geq 0} q^n(q^{n+1})_{n+1},
\end{equation*}    
where 
\begin{equation*}
(a)_n = (a;q)_n \coloneqq \prod_{k=1}^{n} (1-aq^{k-1}).
\end{equation*}
Thus \eqref{LZ0} says that
\begin{equation} \label{LZ}
H(\zeta_N^M) = -\frac{1}{2} \lim_{\tau \to \frac{M}{N}} q^{-\frac{1}{120}} \widetilde{\Phi}_{(2,3,5)}^{(1,1,1)} (\tau),
\end{equation}
where 
\begin{equation*} \label{Hbase}	
H(q) := \sum_{n \geq 0} q^n (q^n)_n.
\end{equation*}

Explicit expressions for $I(\mathcal{M};q)$ are uncommon as they require knowledge of the cyclotomic expansion of the colored Jones polynomial of the knot $K$ from which $\mathcal{M}$ is constructed using Dehn surgery.  One case where this has been fully carried out is for the Brieskorn homology spheres $\Sigma(2,3, 6p-1)$ and $\Sigma(2,3,6p+1)$, which arise from $(-1)$-surgery and $(+1)$-surgery, respectively, along the twist knot $K_p$, $p>0$ \cite[Theorem C.1]{hikami5}.  Specifically, Hikami \cite{hikami5} showed that one has
\begin{equation} \label{h1id}
(1-q) I(\Sigma(2,3,6p-1);q) = H_p^{(1)}(q) 
\end{equation}
for $p>1$ and 
\begin{equation} \label{h8id}
(1-q) I(\Sigma(2,3,6p+1);q) = H_p^{(8)}(q)
\end{equation} 
for $p \geq 1$ where\footnote{Here and throughout, we follow the notation $H_p^{(i)}(q)$ for the $q$-multisums (\ref{h1}) and (\ref{h2})--(\ref{h5}) as given in \cite{tm}. We have labelled (\ref{h8}), (\ref{h6}), (\ref{h7}), (\ref{h9}) and (\ref{h10}) accordingly.} 
\begin{equation} \label{h1}
	H_p^{(1)}(q) := \sum_{n_p \geq \cdots \geq n_1 \geq 0} q^{n_p} (q^{n_p + 1})_{n_p + 1} \prod_{i=1}^{p-1} q^{n_i (n_i + 1)} \begin{bmatrix}
		n_{i+1} \\ n_i \end{bmatrix}
\end{equation}
and 
\begin{equation} \label{h8}
	H_p^{(8)}(q) := \sum_{n_p \geq \cdots \geq n_1 \geq 0} q^{-n_p(n_p+2)} (q^{n_p + 1})_{n_p + 1} \prod_{i=1}^{p-1} q^{n_i (n_i + 1)} \begin{bmatrix}
		n_{i+1} \\ n_i \end{bmatrix}.
\end{equation}
Here
\begin{equation} \label{qbc}
	\begin{bmatrix} n \\ k \end{bmatrix} := \frac{(q)_n}{(q)_{n-k} (q)_k}
\end{equation}
is the $q$-binomial coefficient.   Combining (\ref{h1id}) and (\ref{h8id}) with expressions for the WRT invariants of the Brieskorn homology spheres in terms of Eichler integrals \cite{hikami1}, Hikami deduced the identities \cite[Theorem 3.3]{hikami5}
\begin{equation} \label{h1p}
H_p^{(1)}(\zeta_N) = (1-\zeta_N) I(\Sigma(2,3, 6p-1); \zeta_N) =  -\frac{1}{2} \lim_{\tau \to \frac{1}{N}} q^{-\frac{(6p+5)^2}{24(6p-1)}} \widetilde{\Phi}_{(2,3,6p-1)}^{(1,1,1)} (\tau)
\end{equation}
for $p > 1$ and \cite[Theorem 2.4 and Proposition 3.2]{hikami5}
\begin{equation} \label{h8p}
H_p^{(8)}(\zeta_N) = (1-\zeta_N) I(\Sigma(2,3, 6p+1); \zeta_N) = -\frac{1}{2}  \lim_{\tau \to \frac{1}{N}} q^{1 - \frac{(6p-5)^2}{24(6p+1)}} \widetilde{\Phi}_{(2,3,6p+1)}^{(1,1,1)} (\tau)
\end{equation}
for $p \geq 1$.

In \cite{hikami5}, Hikami conjectured a number of identities for $q$-series at roots of unity which closely resemble (\ref{LZ}), (\ref{h1p}) and (\ref{h8p}) but for which no interpretation in terms of quantum invariants is known. These conjectures are also presented in \cite{tm} and two more were made in private communication \cite{hikami6}.  To state them, we define the following families of $q$-series.   For $p \geq 1$, let
\begin{equation} \label{h2}
H_p^{(2)}(q) := \sum_{n_p \geq \cdots \geq n_1 \geq 0} q^{n_p} (q^{n_p})_{n_p} \prod_{i=1}^{p-1} q^{n_i^2} \begin{bmatrix}
n_{i+1} \\ n_i \end{bmatrix},
\end{equation}
\begin{equation} \label{h3}
H_p^{(3)}(q) := \sum_{n_p \geq \cdots \geq n_1 \geq 0} q^{2n_p} (q^{n_p + 1})_{n_p} \prod_{i=1}^{p-1} q^{n_i (n_i + 1)} \begin{bmatrix}
n_{i+1} \\ n_i \end{bmatrix},
\end{equation}
\begin{equation} \label{h4}
H_p^{(4)}(q) := \sum_{n_p \geq \cdots \geq n_1 \geq 0} q^{n_p} (q^{n_p + 1})_{n_p} \prod_{i=1}^{p-1} q^{n_i (n_i + 1)} \begin{bmatrix}
n_{i+1} \\ n_i \end{bmatrix},
\end{equation}
\begin{equation} \label{h5}
H_p^{(5)}(q) := \sum_{n_p \geq \cdots \geq n_1 \geq 0} q^{n_p} (q^{n_p + 1})_{n_p} \prod_{i=1}^{p-1} q^{n_i^2} \begin{bmatrix}
n_{i+1} \\ n_i \end{bmatrix},
\end{equation}
\begin{equation} \label{h6}
H_p^{(6)}(q) := \sum_{n_p \geq \cdots \geq n_1 \geq 0} q^{-n_p^2}(q^{n_p+1})_{n_p} \prod_{i=1}^{p-1} q^{n_i^2+n_i}\begin{bmatrix} n_{i+1} \\ n_i \end{bmatrix}
\end{equation}
and
\begin{equation} \label{h7}
H_p^{(7)}(q) := \sum_{n_p \geq \cdots \geq n_1 \geq 0} q^{-n_p^2-n_p}(q^{n_p+1})_{n_p} \prod_{i=1}^{p-1} q^{n_i^2+n_i}\begin{bmatrix} n_{i+1} \\ n_i \end{bmatrix}. 
\end{equation}

In  \cite[(3.23), (3.25b)--(3.25d), (3.35)]{hikami5} and \cite{hikami6}, the following conjectures were made.

\begin{conjecture} \label{hconj} We have
\begin{equation} \label{h2conj}
H_2^{(2)}(\zeta_N) = -\frac{1}{2}  \lim_{\tau \to \frac{1}{N}} q^{-\frac{1}{264}} \widetilde{\Phi}_{(2,3,11)}^{(1,1,2)} (\tau),
\end{equation}

\begin{equation} \label{h3conj}
H_2^{(3)}(\zeta_N) = -\frac{1}{2}  \lim_{\tau \to \frac{1}{N}} q^{-(1 + \frac{49}{264})} \widetilde{\Phi}_{(2,3,11)}^{(1,1,3)} (\tau),
\end{equation}

\begin{equation} \label{h41conj}
H_1^{(4)}(\zeta_N) =  -\frac{1}{2}  \lim_{\tau \to \frac{1}{N}} q^{-\frac{49}{120}} \widetilde{\Phi}_{(2,3,5)}^{(1,1,2)} (\tau),
\end{equation}

\begin{equation} \label{h4conj}
H_2^{(4)}(\zeta_N) = -\frac{1}{2}  \lim_{\tau \to \frac{1}{N}} q^{-\frac{169}{264}} \widetilde{\Phi}_{(2,3,11)}^{(1,1,4)} (\tau),
\end{equation}

\begin{equation} \label{h5pconj}
H_p^{(5)}(\zeta_N) = -\frac{1}{2}  \lim_{\tau \to \frac{1}{N}} q^{\frac{1+48p}{24(1-6p)}} \widetilde{\Phi}_{(2,3,6p-1)}^{(1,1,3p-1)} (\tau)
\end{equation}
where $p \geq 1$,
\begin{equation} \label{h6pconj}
H_1^{(6)}(\zeta_N) = -\frac{1}{2} \lim_{\tau \to \frac{1}{N}} q^{-\frac{25}{168}}  \widetilde{\Phi}_{(2,3,7)}^{(1,1,2)}(\tau) 
\end{equation}
and
\begin{equation} \label{h7pconj}
H_1^{(7)}(\zeta_N) = -\frac{1}{2} \lim_{\tau \to \frac{1}{N}} q^{\frac{47}{168}} \widetilde{\Phi}_{(2,3,7)}^{(1,1,3)}(\tau).
\end{equation}
\end{conjecture}

The purpose of this paper is to use the Bailey pair machinery and a novel relation between incomplete quadratic Gauss sums with periodic coefficients to construct identities between each of the infinite families in (\ref{h2})--(\ref{h7}) evaluated at roots of unity and limiting values of Eichler integrals of weight 3/2 modular forms.  This gives new proofs of (\ref{LZ}), (\ref{h1p}) and (\ref{h8p}) as well as a proof and generalization of Conjecture \ref{hconj}.

The Bailey pair method also leads to the discovery of further identities between $q$-series at roots of unity and Eichler integrals.   We will give two examples which involve the series
\begin{equation} \label{h9}
	H_p^{(9)}(q) := \sum_{n_p \geq \cdots \geq n_1 \geq 0} q^{-n_p^2-n_p}(q^{n_p+1})_{n_p} \prod_{i=1}^{p-1} q^{n_i^2} \begin{bmatrix} n_{i+1} \\ n_i \end{bmatrix} 
\end{equation} 
and
\begin{equation} \label{h10}
	H_p^{(10)}(q) := \sum_{n_p \geq \cdots \geq n_1 \geq 0} q^{-n_p^2}(q^{n_p})_{n_p} \prod_{i=1}^{p-1} q^{n_i^2} \begin{bmatrix} n_{i+1} \\ n_i \end{bmatrix},
\end{equation}
defined for $p \geq 1$.   We now state our main result.

\begin{theorem} \label{main} Let $M$ and $N$ be coprime positive integers. For $p \geq 2$, we have
\begin{equation} \label{main1eq}
H_p^{(1)}(\zeta_N^M) = -\frac{1}{2}  \lim_{\tau \to \frac{M}{N}} q^{-\frac{(6p+5)^2}{24(6p-1)}}\widetilde{\Phi}_{(2,3,6p-1)}^{(1,1,1)} (\tau)
\end{equation}
while for $p=1$
\begin{equation} \label{main11eq}
H_1^{(1)}(\zeta_N^M) = -\zeta_N^{-M} - \frac{1}{2}  \lim_{\tau \to \frac{M}{N}} q^{-\frac{121}{120} }\widetilde{\Phi}_{(2,3,5)}^{(1,1,1)} (\tau).
\end{equation}  
For $p \geq 1$, we have
\begin{equation} \label{main2eq}
H_p^{(2)}(\zeta_N^M) = -\frac{1}{2}  \lim_{\tau \to \frac{M}{N}} q^{-\frac{1}{24(6p-1)}} \widetilde{\Phi}_{(2,3,6p-1)}^{(1,1,p)} (\tau),
\end{equation}
\begin{equation} \label{main3eq}
H_p^{(3)}(\zeta_N^M) = -\frac{1}{2}  \lim_{\tau \to \frac{M}{N}} q^{-1 - \frac{(6p-5)^2}{24(6p-1)}} \widetilde{\Phi}_{(2,3,6p-1)}^{(1,1,2p-1)} (\tau),
\end{equation}
\begin{equation} \label{main4eq}
H_p^{(4)}(\zeta_N^M) = -\frac{1}{2}  \lim_{\tau \to \frac{M}{N}} q^{-\frac{(6p+1)^2}{24(6p-1)}} \widetilde{\Phi}_{(2,3,6p-1)}^{(1,1,2p)} (\tau),
\end{equation}
\begin{equation} \label{main5eq}
H_p^{(5)}(\zeta_N^M) = -\frac{1}{2}  \lim_{\tau \to \frac{M}{N}} q^{p-1 - \frac{(12p-5)^2}{24(6p-1)}} \widetilde{\Phi}_{(2,3,6p-1)}^{(1,1,3p-1)} (\tau),
\end{equation}
\begin{equation} \label{main6}
H_p^{(6)}(\zeta_N^M) = -\frac{1}{2}  \lim_{\tau \to \frac{M}{N}} q^{-\frac{(6p-1)^2}{24(6p+1)}} \widetilde{\Phi}_{(2,3,6p+1)}^{(1,1,2p)} (\tau),
\end{equation}
\begin{equation} \label{main7eq}
H_p^{(7)}(\zeta_N^M) = -\frac{1}{2}  \lim_{\tau \to \frac{M}{N}} q^{1 - \frac{(6p+5)^2}{24(6p+1)}} \widetilde{\Phi}_{(2,3,6p+1)}^{(1,1,2p+1)} (\tau),
\end{equation}
\begin{equation} \label{main8eq}
H_p^{(8)}(\zeta_N^M) = -\frac{1}{2}  \lim_{\tau \to \frac{M}{N}} q^{1 - \frac{(6p-5)^2}{24(6p+1)}} \widetilde{\Phi}_{(2,3,6p+1)}^{(1,1,1)} (\tau),
\end{equation}
\begin{equation} \label{main9eq}
H_p^{(9)}(\zeta_N^M) = -\frac{1}{2}  \lim_{\tau \to \frac{M}{N}} q^{p-1 - \frac{(12p-1)^2}{24(6p+1)}} \widetilde{\Phi}_{(2,3,6p+1)}^{(1,1,3p)} (\tau)
\end{equation}
and
\begin{equation} \label{main10eq}
H_p^{(10)}(\zeta_N^M) =  -\frac{1}{2}  \lim_{\tau \to \frac{M}{N}} q^{-\frac{1}{24(6p+1)}} \widetilde{\Phi}_{(2,3,6p+1)}^{(1,1,p)} (\tau).
\end{equation}
\end{theorem}

\begin{remark}
Observe that the $p=1$ case of (\ref{main2eq}) recovers (\ref{LZ}) while the $M=1$ cases of (\ref{main1eq}) and (\ref{main8eq}) yield (\ref{h1p}) and (\ref{h8p}), respectively. We have also included (\ref{main11eq}) for completeness (cf. \cite[(3.8), (3.10)]{hikami5}). Note that Theorem \ref{main} embeds Hikami's conjectures for one single value of $p$ (with the exception of (\ref{h5pconj})) into an infinite family of identities. Finally, our main result and \cite[Section 4.4]{go} imply that each $H_p^{(i)}(q)$ is a quantum modular form.
\end{remark}

The paper is organized as follows. In Section 2, we recall the necessary background on the Bailey pair machinery and express all of the series $H_p^{(i)}(q)$ as finite polynomial sums at roots of unity.  See Theorems \ref{Baileyfamily1}--\ref{Baileyfamily10}.   An elucidatory principle in this section arises from the broader perspective of quantum $q$-series identities \cite{love1, love2}, a recent development whose applications include new proofs for identities at roots of unity for Ramanujan's $\sigma$-function and its companion $\sigma^{*}(q)$ \cite{ADH, Cohen}, a ``duality" result for the Kontsevich-Zagier strange function \cite{bopr, fkvy} and Zagier's strange identity and Hikami's generalization thereof \cite{hikamiRAMA, z}. For further work in this direction, see \cite{dl, fm, f, ls}. In Section 3, we prove key properties for and relations between incomplete quadratic Gauss sums. In Section 4, we convert the finite polynomial sums in Section 2 to limiting values of Eichler integrals, thereby proving Theorem \ref{main}.

\section{Bailey pairs and key identities}

\subsection{Preliminaries and auxiliary results}

Recall that a Bailey pair relative to $a$ is a pair of sequences $(\alpha_n,\beta_n)_{n \geq 0}$ satisfying
\begin{align} 
\beta_n &= \sum_{k=0}^n \frac{\alpha_k}{(q)_{n-k}(aq)_{n+k}} \label{pairdef} \\
&= \frac{1}{(q)_n(aq)_n} \sum_{k=0}^n \frac{(q^{-n})_k}{(aq^{n+1})_k}(-1)^kq^{nk -\binom{k}{2}} \alpha_k. \label{pairdefbis}
\end{align}
The first equation is the original definition \cite{An1}, while the second follows from an application of the identity \cite[Appendix I, Eq. (I.10)]{Ga-Ra}
\begin{equation} \label{qPochn-k}
(x)_{n-k} = \frac{(x)_n}{(q^{1-n}/x)_k}(-q/x)^kq^{\binom{k}{2} - nk}.
\end{equation} 

We begin by reviewing two important facts about Bailey pairs.   First, we require a special case of the Bailey lemma \cite{An1}.
\begin{lemma} \label{Baileylemma}
If $(\alpha_n,\beta_n)$ is a Bailey pair relative to $a$, then so is $(\alpha_n',\beta_n')$, where
\begin{equation} \label{alphaprime}
\alpha_n' = a^nq^{n^2}\alpha_n 
\end{equation}
and
\begin{equation} \label{betaprime}
\beta_n' = \sum_{j=0}^n \frac{a^jq^{j^2}}{(q)_{n-j}} \beta_j.
\end{equation}
\end{lemma} 
Second, we have a special case of the Bailey lattice \cite{Ag-An-Br1}.
\begin{lemma} \label{Baileylattice}
If $(\alpha_n,\beta_n)$ is a Bailey pair relative to $q^2$, then $(\alpha_n^*,\beta_n)$ is a Bailey pair relative to $q$, where
\begin{equation} \label{alphastar}
\alpha_n^* = (1-q^2)\left(\frac{\alpha_n}{1-q^{2n+2}} - \frac{q^{2n}\alpha_{n-1}}{1-q^{2n}} \right).
\end{equation} 
By convention we take $\alpha_{-1} = 0 $.
\end{lemma}

Using Lemmas \ref{Baileylemma} and \ref{Baileylattice}, we establish three auxiliary results which will allow us to relate the $q$-multisums in Theorem \ref{main} to polynomial sums at roots of unity.  The first result is for Bailey pairs relative to $q$.

\begin{proposition} \label{keychainlemma}
If $(\alpha_k,\beta_k)$ is a Bailey pair relative to $q$, then for any $p \geq 1$ we have
\begin{equation} \label{keychainlemmaeq}
\begin{aligned}
\sum_{n-1 \geq n_p \geq \cdots \geq n_1\geq 0}& \frac{q^{(n-n_p-1)(n-n_p)}\beta_{n-n_p-1}}{(q)_{n_p}} \prod_{i=1}^{p-1} q^{(n-n_i-1)(n-n_i)} \begin{bmatrix} n_{i+1} \\ n_i \end{bmatrix} \\
&= \frac{1}{(q)_{n-1}(q^2)_{n-1}} \sum_{k=0}^{n-1} \frac{(q^{1-n})_k}{(q^{1+n})_k}(-1)^kq^{nk -\binom{k+1}{2} + p(k^2+k)}\alpha_k.
\end{aligned}
\end{equation}
\end{proposition}

\begin{proof}
	Suppose that $(\alpha_k,\beta_k)$ is a Bailey pair relative to $q$.   Iterating \eqref{alphaprime} and \eqref{betaprime} $p$ times we obtain another Bailey pair relative to $q$,
	\begin{equation} \label{alpha^p}
		\alpha_k^{(p)} = q^{p(k^2+k)}\alpha_k 
	\end{equation}
	and
	\begin{equation}
		\beta_k^{(p)} = \sum_{k \geq n_p \geq \cdots \geq n_1\geq 0} \frac{q^{n_p(n_p+1) + \cdots + n_1(n_1+1)}}{(q)_{k-n_p} \cdots (q)_{n_2-n_1}} \beta_{n_1}.
	\end{equation}
	Now in $\beta_k^{(p)}$ we replace $k$ by $n-1$ and then make the change of indices $n_i \rightarrow n - 1 - n_{p-i+1}$ for $1 \leq i \leq p$.   The result is
	\begin{equation}
		\beta_{n-1}^{(p)} = \sum_{n-1 \geq n_p \geq \cdots \geq n_1\geq 0} \frac{q^{(n - n_1 - 1)(n- n_1) + \cdots + (n-n_p-1)(n-n_p)}}{(q)_{n_1} (q)_{n_2-n_1} \cdots (q)_{n_p-n_{p-1}}} \beta_{n -n_p -1}.
	\end{equation}
	Multiplying numerator and denominator by $(q)_{n_p}$ and then converting to $q$-binomial coefficients (\ref{qbc}) gives the left-hand side of \eqref{keychainlemmaeq}.   The right-hand side follows from inserting \eqref{alpha^p} into \eqref{pairdefbis}.    This completes the proof.
\end{proof}

Our second auxiliary result is analogous to Proposition \ref{keychainlemma}, but for Bailey pairs relative to $q^2$.  In what follows, we define $a_n^*$ for any sequence $a_n$ by
\begin{equation} \label{astar}
	a_n^* = (1-q^2)\left(\frac{a_n}{1-q^{2n+2}} - \frac{q^{2n}a_{n-1}}{1-q^{2n}} \right)
\end{equation}  
in analogy with \eqref{alphastar}. 
 
\begin{proposition} \label{keychainlemmabis}
If $(\alpha_k,\beta_k)$ is a Bailey pair relative to $q^2$, then for any $p \geq 1$ we have
\begin{equation} \label{keychainlemmabiseq}
\begin{aligned}
\sum_{n-1 \geq n_p \geq \cdots \geq n_1\geq 0}& \frac{q^{(n-n_p-1)(n-n_p+1)}\beta_{n-n_p-1}}{(q)_{n_p}} \prod_{i=1}^{p-1} q^{(n-n_i-1)(n-n_i+1)} \begin{bmatrix} n_{i+1} \\ n_i \end{bmatrix} \\
&= \frac{1}{(q)_{n-1}(q^2)_{n-1}} \sum_{k=0}^{n-1} \frac{(q^{1-n})_k}{(q^{1+n})_k}(-1)^kq^{nk -\binom{k+1}{2}} (q^{p(k^2+2k)} \alpha_k)^*.
\end{aligned}
\end{equation} 
\end{proposition}

\begin{proof}
	Suppose that $(\alpha_k,\beta_k)$ is a Bailey pair relative to $q^2$.   We argue as in Proposition \ref{keychainlemma} with the same iterations of \eqref{alphaprime} and \eqref{betaprime} (but with $a=q^2$) and the same changes of indices.  The result is a Bailey pair relative to $q^2$, 
	\begin{equation} \label{iteratedalphaq^2}
		\alpha_k^{(p)} = q^{p(k^2+2k)}\alpha_k 
	\end{equation}  
	and 
	\begin{equation} \label{iteratedbetaq^2}
		\beta_{n-1}^{(p)} = \sum_{n-1 \geq n_p \geq \cdots \geq n_1\geq 0} \frac{q^{(n - n_1 - 1)(n- n_1 +1) + \cdots + (n-n_p-1)(n-n_p+1)}}{(q)_{n_1} (q)_{n_2-n_1} \cdots (q)_{n_p-n_{p-1}}} \beta_{n -n_p +1}.
	\end{equation}
	Now we change this to a Bailey pair relative to $q$ using Lemma \ref{Baileylattice} and then insert it into \eqref{pairdefbis} to complete the proof.    
\end{proof}

The third auxiliary result also applies to Bailey pairs relative to $q^2$, but in a different way.

\begin{proposition} \label{keychainlemmater}
	If $(\alpha_k,\beta_k)$ is a Bailey pair relative to $q^2$, then for any $p \geq 1$ we have
	\begin{equation} \label{keychainlemmatereq}
		\begin{aligned}
			\sum_{n-1 \geq n_p \geq \cdots \geq n_1\geq 0}& \frac{q^{(n-n_p-1)(n-n_p+1)}\beta_{n-n_p-1}}{(q)_{n_p}} \prod_{i=1}^{p-1} q^{(n-n_i-1)(n-n_i+1)} \begin{bmatrix} n_{i+1} \\ n_i \end{bmatrix} \\
			&= \frac{(1-q^2)}{(q)_{n-1}(q)_{n}} \sum_{k=0}^{n-1} \frac{(q^{1-n})_k}{(q^{2+n})_k}(-1)^k q^{nk -\binom{k+1}{2} + p(k^2+2k)} \alpha_k.
		\end{aligned}
	\end{equation}
\end{proposition}

\begin{proof}
	This follows from substituting \eqref{iteratedalphaq^2} and \eqref{iteratedbetaq^2} into the definition of a Bailey pair in \eqref{pairdefbis} with $a=q^2$.
\end{proof}

\subsection{An arsenal of Bailey pairs} \label{arseBP}
Here we collect all of the Bailey pairs we will use in the proof of Theorem \ref{main}.   There are eight Bailey pairs that we cite from the literature -- six Bailey pairs relative to $q$ and two Bailey pairs relative to $q^2$.   First, we have four Bailey pairs relative to $q$ from Slater's list \cite[p.463, A(2), A(4), A(6), A(8)]{Sl}:

\begin{equation} \label{BP1}
	\alpha_k = 
	\begin{cases}
		q^{6r^2 -r} & \text{if $k=3r-1$}, \\
		q^{6r^2 +r} & \text{if $k=3r$}, \\
		-q^{6r^2+5r+1} - q^{6r^2+7r+2} &\text{if $k=3r+1$}
	\end{cases}
	\ \ \ \ \text{and} \ \ \ \ \beta_k = \frac{1}{(q^2)_{2k}},
\end{equation} 

\begin{equation} \label{BP2}
	\alpha_k = 
	\begin{cases}
		q^{6r^2 -4r} & \text{if $k=3r-1$}, \\
		q^{6r^2 +4r} & \text{if $k=3r$}, \\
		-q^{6r^2+8r+2} - q^{6r^2+4r} &\text{if $k=3r+1$}
	\end{cases}
\ \ \ \ \text{and} \ \ \ \ \beta_k = \frac{q^k}{(q^2)_{2k}},
\end{equation}  

\begin{equation} \label{BP3}
	\alpha_k = 
	\begin{cases}
		q^{3r^2+r} & \text{if $k=3r-1$}, \\
		q^{3r^2-r} & \text{if $k=3r$}, \\
		-q^{3r^2+r} - q^{3r^2+5r+2} &\text{if $k=3r+1$}
	\end{cases}
\ \ \ \ \text{and} \ \ \ \ \beta_k = \frac{q^{k^2}}{(q^2)_{2k}}
\end{equation}
and
\begin{equation} \label{BP4}
	\alpha_k = 
	\begin{cases}
		q^{3r^2 -2r} & \text{if $k=3r-1$}, \\
		q^{3r^2 +2r} & \text{if $k=3r$}, \\
		-q^{3r^2+4r+1} - q^{3r^2+2r} &\text{if $k=3r+1$}
	\end{cases}
\ \ \ \ \text{and} \ \ \ \ \beta_k = \frac{q^{k^2+k}}{(q^2)_{2k}}.
\end{equation}

Next we have two Bailey pairs relative to $q$ from work of Warnaar \cite{Wa}:

\begin{equation} \label{BP5}
	\alpha_k = (-1)^{\lfloor \frac{4k+1}{3} \rfloor}q^{\frac{k(2k-1)}{3}}\frac{(1-q^{2k+1})}{(1-q)}\chi(k \not \equiv 1 \pmod{3})
\ \ \ \ \text{and} \ \ \ \	\beta_k = \frac{1}{(q)_{2k}}
\end{equation}
and
\begin{equation} \label{BP6}
	\alpha_k = (-1)^{\lfloor \frac{4k+1}{3} \rfloor}q^{\frac{k(k-2)}{3}}\frac{(1-q^{2k+1})}{(1-q)}\chi(k \not \equiv 1 \pmod{3}) \ \ \ \ \text{and} \ \ \ \ 
	\beta_k = \frac{q^{k^2-k}}{(q)_{2k}}.
\end{equation}

Finally, we have two Bailey pairs relative to $q^2$.   These are \cite[Eq. (2.34), $e \to 0$ and $e \to \infty$]{Si-Mc-Zi}:
\begin{equation} \label{BP7}
	\alpha_k = 
	\begin{dcases}
	        0 &\text{if $k=3r-1$}, \\	
		\frac{(1-q^{6r+2})}{1-q^2}q^{3r^2-r} & \text{if $k=3r$}, \\
		-\frac{(1-q^{6r+4})}{1-q^2}q^{3r^2+r} & \text{if $k=3r+1$} 
	\end{dcases}
\ \ \ \ \text{and} \ \ \ \ 
	\beta_k = \frac{q^{k^2}}{(q^2)_{2k}}
\end{equation}
and
\begin{equation} \label{BP8} 
	\alpha_k = 
	\begin{dcases}
	         0 &\text{if $k=3r-1$}, \\
		\frac{(1-q^{6r+2})}{1-q^2}q^{6r^2+r} & \text{if $k=3r$}, \\
		-\frac{(1-q^{6r+4})}{1-q^2}q^{6r^2+5r+1} & \text{if $k=3r+1$}
	\end{dcases}
\ \ \ \ \text{and} \ \ \ \
	\beta_k = \frac{1}{(q^2)_{2k}}.
\end{equation}

To these we add two Bailey pairs relative to $q^2$ which we were not able to locate in the literature.
\begin{lemma}
The following two sequences are Bailey pairs relative to $q^2$:
	\begin{equation} \label{BP9}
		\alpha_k = 
		\begin{dcases}
			\frac{(1-q^{k+1})}{1-q}\left(q^{3r^2-2r} - q^{3r^2-r}\right) & \text{if $k=3r-1$}, \\
			\frac{(1-q^{k+1})}{1-q}q^{3r^2+r} & \text{if $k=3r$}, \\
			-\frac{(1-q^{k+1})}{1-q}q^{3r^2+2r} &\text{if $k=3r+1$}
		\end{dcases}
		\ \ \ \ \text{and} \ \ \ \ \beta_k = \frac{q^{k^2+k}(1+q)}{(q^2)_{2k}(1+q^{k+1})}
	\end{equation}
and	
	\begin{equation} \label{BP10}
		\alpha_k = 
		\begin{dcases}
			\frac{(1-q^{k+1})}{1-q}\left(q^{6r^2-r} - q^{6r^2-2r}\right) & \text{if $k=3r-1$}, \\
			\frac{(1-q^{k+1})}{1-q}q^{6r^2+2r} & \text{if $k=3r$}, \\
			-\frac{(1-q^{k+1})}{1-q}q^{6r^2+7r+2} &\text{if $k=3r+1$}
		\end{dcases}
		\ \ \ \ \text{and} \ \ \ \ \beta_k = \frac{(1+q)}{(q^2)_{2k}(1+q^{k+1})}.
	\end{equation}  
	
\end{lemma}

\begin{proof}
	We require \cite[Eq. (2.4)]{Lo.5}, which says that if $(\alpha_k,\beta_k)$ is a Bailey pair relative to $a$, then $(\alpha_k^{\dag},\beta_k^{\dag})$ is a Bailey pair relative to $aq$, where
	\begin{equation} \label{alphadag}
		\alpha_k^{\dag} = \frac{(1-aq^{2k+1})(aq/b)_k(-b)^kq^{\binom{k}{2}}}{(1-aq)(bq)_k} \sum_{j=0}^k \frac{(b)_j}{(aq/b)_j} (-b)^{-j} q^{-\binom{j}{2}} \alpha_j
	\end{equation}
	and
	\begin{equation} \label{betadag}
		\beta_k^{\dag} = \frac{(b)_k}{(bq)_k}\beta_k.
	\end{equation}
	We apply this with $b = -q$ to the Bailey pair relative to $q$ in \eqref{BP4}.  The resulting $\beta_k^{\dag}$ is clearly the $\beta_k$ in \eqref{BP9}, while
	\begin{equation}
		\alpha_k^{\dag} = \frac{(1-q^{k+1})}{1-q} q^{\binom{k+1}{2}} \sum_{j=0}^k q^{-\binom{j+1}{2}}\alpha_j.
	\end{equation}  
	An induction argument on $k$ confirms that $\alpha_k^{\dag}$ equals $\alpha_k$ from \eqref{BP9}. 
	
Next we recall that if $(\alpha_k,\beta_k)$ is a Bailey pair relative to $q^2$, then its dual \cite{An1} $(\alpha_k^d,\beta_k^d)$ is also a Bailey pair relative to $q^2$, where
\begin{equation}
\alpha_k^d = q^{k^2+2k}\alpha_k(q^{-1}) \ \ \ \text{and} \ \ \ \beta_k^d = q^{-k^2-3k}\beta_k(q^{-1}).
\end{equation}	
Using the fact that
\begin{equation}
(q^{-2};q^{-1})_{2k} = q^{-2k^2-3k}(q^2)_{2k},
\end{equation}
a short calculation shows that \eqref{BP10} is the dual of \eqref{BP9}.   
\end{proof}

\subsection{The key results}
Here we state and prove the key results relating (\ref{h1}), (\ref{h8}), (\ref{h2})--(\ref{h7}), (\ref{h9}) and (\ref{h10}) to finite polynomial sums at roots of unity. We have chosen to present these results following the order in which the Bailey pairs in Section \ref{arseBP} were introduced.
\begin{theorem} \label{Baileyfamily1}
	Let $(\alpha_k,\beta_k)$ denote the Bailey pair in \eqref{BP1}.   Then for $p \geq 1$ and $q$ any primitive $N$th root of unity we have
	\begin{equation} \label{Baileyfamily1eq}
		H_p^{(6)}(q) = \sum_{k=0}^{N-1} (-1)^k q^{\binom{k+1}{2} +(p-1)(k^2+k)} \alpha_k.
	\end{equation} 
\end{theorem}

\begin{theorem} \label{Baileyfamily2}
	Let $(\alpha_k,\beta_k)$ denote the Bailey pair in \eqref{BP2}.   Then for $p \geq 1$ and $q$ any primitive $N$th root of unity we have
	\begin{equation} \label{Baileyfamily2eq}
		H_p^{(7)}(q)  = q\sum_{k=0}^{N-1} (-1)^k q^{\binom{k+1}{2} +(p-1)(k^2+k)} \alpha_k.
	\end{equation} 
\end{theorem}

\begin{theorem} \label{Baileyfamily3} 
	Let $(\alpha_k,\beta_k)$ denote the Bailey pair in \eqref{BP3}.   Then for $p \geq 1$ and $q$ any primitive $N$th root of unity we have
	\begin{equation} \label{Baileyfamily3eq}
		H_p^{(3)}(q) = q^{-1}\sum_{k=0}^{N-1} (-1)^k q^{\binom{k+1}{2} +(p-1)(k^2+k)} \alpha_k.
	\end{equation} 
\end{theorem}

\begin{theorem} \label{Baileyfamily4} 
	Let $(\alpha_k,\beta_k)$ denote the Bailey pair in \eqref{BP4}.   Then for $p \geq 1$ and $q$ any primitive $N$th root of unity we have
	\begin{equation} \label{Baileyfamily4eq}
		H_p^{(4)}(q)  = \sum_{k=0}^{N-1} (-1)^k q^{\binom{k+1}{2} + (p-1)(k^2+k)} \alpha_k.
	\end{equation} 
\end{theorem}

\begin{theorem} \label{Baileyfamily5} 
	Let $(\alpha_k,\beta_k)$ denote the Bailey pair in \eqref{BP5}.   Then for $p \geq 1$ and $q$ any primitive $N$th root of unity we have
	\begin{equation} \label{Baileyfamily5eq}
		H_p^{(8)}(q) = -q(1-q)\sum_{k=0}^{N-1} (-1)^k q^{\binom{k+1}{2} +(p-1)(k^2+k)} \alpha_k.
	\end{equation} 
\end{theorem}

\begin{theorem} \label{Baileyfamily6}
	Let $(\alpha_k,\beta_k)$ denote the Bailey pair in \eqref{BP6}.   Then for $p \geq 1$ and $q$ any primitive $N$th root of unity we have
	\begin{equation} \label{Baileyfamily6eq}
		H_p^{(1)}(q)  = -q^{-1}(1-q)\sum_{k=0}^{N-1} (-1)^k q^{\binom{k+1}{2} +(p-1)(k^2+k)} \alpha_k.
	\end{equation} 
\end{theorem} 

\begin{theorem} \label{Baileyfamily7}
	Let $(\alpha_k,\beta_k)$ denote the Bailey pair in \eqref{BP7}.   For $p \geq 1$ and $q$ any primitive $N$th root of unity we have
	\begin{equation} \label{Baileyfamily7eq}
		H_p^{(4)}(q)  = q^{p-1}\sum_{k=0}^{N-1} (-1)^k q^{-\binom{k+1}{2}} (q^{p(k^2+2k)} \alpha_k)^*.
	\end{equation}
\end{theorem} 

\begin{theorem} \label{Baileyfamily8}
	Let $(\alpha_k,\beta_k)$ denote the Bailey pair in \eqref{BP8}.   For $p \geq 1$ and $q$ any primitive $N$th root of unity we have
	\begin{equation} \label{Baileyfamily8eq}
		H_p^{(9)}(q) = q^{p}\sum_{k=0}^{N-1} (-1)^k q^{-\binom{k+1}{2}} (q^{p(k^2+2k)} \alpha_k)^*.
	\end{equation}
\end{theorem}
For the last two results we define
\begin{equation} \label{alphatilde}
\widehat{\alpha}_k = \frac{(1-q)}{1-q^{k+1}}\alpha_k.
\end{equation}

\begin{theorem} \label{Baileyfamily9}
	Let $(\alpha_k,\beta_k)$ denote the Bailey pair in \eqref{BP9}.  For $p \geq 1$ and $q$ any primitive $N$th root of unity we have
	\begin{equation} \label{Baileyfamily9eq}
		\begin{aligned}
			-\frac{1}{2} + H_p^{(2)}(q)
			&= q^{p}\left(\sum_{k=0}^{N-2} (-1)^{k} q^{-\binom{k+1}{2} + p(k^2+2k)} \widehat{\alpha}_k + \frac{1}{2} q^{-p}\widehat{\alpha}_{N-1}\right).
		\end{aligned}
	\end{equation}
\end{theorem}

\begin{theorem} \label{Baileyfamily10}
	Let $(\alpha_k,\beta_k)$ denote the Bailey pair in \eqref{BP10}.  For $p \geq 1$ and $q$ any primitive $N$th root of unity we have
	\begin{equation} \label{Baileyfamily10eq}
		\begin{aligned}
			-\frac{1}{2} + H_p^{(10)}(q)
			&= q^{p}\left(\sum_{k=0}^{N-2} (-1)^{k} q^{-\binom{k+1}{2} + p(k^2+2k)} \widehat{\alpha}_k + \frac{1}{2} q^{-p}\widehat{\alpha}_{N-1}\right).
		\end{aligned}
	\end{equation}
\end{theorem}

\begin{proof}[Proof of Theorem \ref{Baileyfamily1}]
Let $(\alpha_k,\beta_k)$ denote the Bailey pair in \eqref{BP1}.   Using \eqref{qPochn-k} we compute that
\begin{equation} \label{betan-1-p1}
\beta_{N-1-n_p} = \frac{(q^{1-2N})_{2n_p} q^{4Nn_p-2n_p^2-n_p}}{(q^2)_{2N-2}}.
\end{equation}
Inserting this into \eqref{keychainlemmaeq} we obtain
\begin{equation}
\begin{aligned}
\sum_{N-1 \geq n_p \geq \cdots \geq n_1\geq 0}& \frac{q^{(N-n_p-1)(N-n_p)}(q^{1-2N})_{2n_p} q^{4Nn_p-2n_p^2-n_p}}{(q)_{n_p}} \prod_{i=1}^{p-1} q^{(N-n_i-1)(N-n_i)} \begin{bmatrix} n_{i+1} \\ n_i \end{bmatrix} \\
&= \frac{(q^2)_{2N-2}}{(q)_{N-1}(q^2)_{N-1}} \sum_{k=0}^{N-1} \frac{(q^{1-N})_k}{(q^{1+N})_k}(-1)^k q^{Nk -\binom{k+1}{2} + p(k^2+k)} \alpha_k.
\end{aligned}
\end{equation}
Now we let $q$ be a primitive $N$th root of unity.   The terms $\frac{(q^2)_{2N-2}}{(q)_{N-1}(q^2)_{N-1}}$ and $\frac{(q^{1-N})_k}{(q^{1+N})_k}$ become $1$, as do all other instances of $q^N$.   The result is \eqref{Baileyfamily1eq}.
\end{proof}

\begin{proof}[Proofs of Theorems \ref{Baileyfamily2}--\ref{Baileyfamily6}] 
The proofs of Theorems \ref{Baileyfamily2}--\ref{Baileyfamily6} closely resemble the proof of Theorem \ref{Baileyfamily1}.   The only difference is the initial Bailey pair $(\alpha_k,\beta_k)$.  Therefore we limit ourselves to noting the expression for $\beta_{N-1-n_p}$ in each case.   For the Bailey pairs in \eqref{BP2}--\eqref{BP6}, we have, respectively,
\begin{equation}
\beta_{N-1-n_p} = \frac{(q^{1-2N})_{2n_p} q^{4Nn_p-2n_p^2-2n_p + N-1}}{(q^2)_{2N-2}},
\end{equation} 
\begin{equation} \label{betan-1-p3}
\beta_{N-1-n_p} = \frac{(q^{1-2N})_{2n_p} q^{4Nn_p-2n_p^2-n_p + (N-n_p-1)^2}}{(q^2)_{2N-2}},
\end{equation} 
\begin{equation}
\beta_{N-1-n_p} = \frac{(q^{1-2N})_{2n_p} q^{4Nn_p-2n_p^2-n_p + (N-n_p-1)^2 + (N-n_p-1)}}{(q^2)_{2N-2}},
\end{equation} 
\begin{equation}
\beta_{N-1-n_p} = \frac{(q^{2-2N})_{2n_p} q^{4Nn_p-2n_p^2-3n_p}}{(q)_{2N-2}}
\end{equation}
and
\begin{equation}
\beta_{N-1-n_p} = \frac{(q^{2-2N})_{2n_p} q^{4Nn_p-2n_p^2-3n_p + (N-1-n_p)^2 - (N-1-n_p)}}{(q)_{2N-2}}.
\end{equation}
Using these in \eqref{keychainlemma} and calculating as in the proof of Theorem \ref{Baileyfamily1} we obtain \eqref{Baileyfamily2eq}--\eqref{Baileyfamily6eq}.
\end{proof}

The proofs of Theorems \ref{Baileyfamily7} and \ref{Baileyfamily8} are similar, the only difference being the use of Proposition \ref{keychainlemmabis} instead of Proposition \ref{keychainlemma}.

\begin{proof}[Proofs of Theorems \ref{Baileyfamily7}--\ref{Baileyfamily8}]
Let $(\alpha_k,\beta_k)$ denote the Bailey pair in \eqref{BP7}.   Note that the $\beta_k$ is the same as in \eqref{BP3}, and so $\beta_{N-1-n_p}$ is given by \eqref{betan-1-p3}.   Using this in \eqref{keychainlemmabiseq} we have
\begin{equation}
\begin{aligned}
\sum_{N-1 \geq n_p \geq \cdots \geq n_1\geq 0}& \frac{q^{(N-n_p)^2-1} (q^{1-2N})_{2n_p} q^{4Nn_p-2n_p^2-n_p + (N-n_p-1)^2}}{(q)_{n_p}} \prod_{i=1}^{p-1} q^{(N-n_i-1)(N-n_i+1)} \begin{bmatrix} n_{i+1} \\ n_i \end{bmatrix} \\
&= \frac{(q^2)_{2N-2}}{(q)_{N-1}(q^2)_{N-1}} \sum_{k=0}^{N-1} \frac{(q^{1-N})_k}{(q^{1+N})_k}(-1)^k q^{Nk -\binom{k+1}{2}} (q^{p(k^2+2k)} \alpha_k)^*.
\end{aligned}
\end{equation} 	
Letting $q$ be a primitive $N$th root of unity gives \eqref{Baileyfamily7eq}.   Theorem \ref{Baileyfamily8} follows in the same way using the Bailey pair in \eqref{BP8}.
\end{proof}

\begin{proof}[Proofs of Theorems \ref{Baileyfamily9}--\ref{Baileyfamily10}]
For Theorem \ref{Baileyfamily9} we use \eqref{BP9} in Proposition \ref{keychainlemmater} to obtain
\begin{equation} \label{obtained9}
\begin{aligned}
\sum_{N-1 \geq n_p \geq \cdots \geq n_1\geq 0}& \frac{q^{(N-n_p-1)(N-n_p+1)} q^{2Nn_p - n_p^2 + N^2-N}(q^{1-2N})_{2n_p}}{(1+q^{N-n_p})(q)_{n_p}} \prod_{i=1}^{p-1} q^{(N-n_i-1)(N-n_i+1)} \begin{bmatrix} n_{i+1} \\ n_i \end{bmatrix} \\
&= \frac{(q)_{2N-1}}{(q)_{N-1}(q)_{N}} \sum_{k=0}^{N-1} \frac{(q^{1-N})_k}{(q^{2+N})_k}(-1)^k q^{Nk -\binom{k+1}{2} + p(k^2+2k)} \alpha_k.
\end{aligned}
\end{equation} 
Letting $q$ be a primitive $N$th root of unity the left-hand side of (\ref{obtained9}) becomes 
\begin{equation} \label{next}
q^{-p}\sum_{n_p \geq \cdots \geq n_1 \geq 0} \frac{q^{n_p}(q^{n_p+1})_{n_p}}{(1+q^{n_p})} \prod_{i=1}^{p-1} q^{n_i^2} \begin{bmatrix} n_{i+1} \\ n_i \end{bmatrix}.
\end{equation}
Note that
\begin{equation}
\frac{(q^{n_p+1})_{n_p}}{(1+q^{n_p})} = 
\begin{dcases}
(q^{n_p})_{n_p} & \text{if $n_p > 0$} \\
\frac{1}{2} & \text{if $n_p = 0$}
\end{dcases}
\end{equation}
and so (\ref{next}) when $q$ is a primitive $N$th root of unity is
\begin{equation}
q^{-p}\left(-\frac{1}{2} + \sum_{n_p \geq \cdots \geq n_1 \geq 0} q^{n_p}(q^{n_p})_{n_p} \prod_{i=1}^{p-1} q^{n_i^2} \begin{bmatrix} n_{i+1} \\ n_i \end{bmatrix} \right).
\end{equation}

Now letting $q$ be a primitive $N$th root of unity on the right-hand side of \eqref{obtained9}, the prefactor becomes $1$ as usual.   As for the term $\frac{(q^{1-N})_k}{(q^{2+N})_k}$, it does not become $1$ as before.  But when multiplied by $(1-q^{k+1})/(1-q)$, it becomes $1$ for $k < N-1$ while for $k = N-1$ it is $\frac{(1-q^N)}{(1-q^{2N})} \to \frac{1}{2}$.   This observation accounts for the use of $\widehat{\alpha}_k$ and the splitting of the sum on the right-hand side of \eqref{Baileyfamily9eq}. Here, we also use that $q^{-\binom{N}{2} + pN^2} = (-1)^{N+1}$. This gives Theorem \ref{Baileyfamily9}.    Theorem \ref{Baileyfamily10} follows in the same way using \eqref{BP10} in Proposition \ref{keychainlemmater}.
\end{proof}

\section{Incomplete quadratic Gauss sums}

\subsection{Properties for incomplete quadratic Gauss sums}

We begin with a general result on an incomplete quadratic Gauss sum. For simplicity, we write $\chi = \chi_{\bm{p}}^{\bm{\ell}}$. Recall that $P:=p_1 p_2 p_3$.

\begin{lemma} \label{chiprop} Assume that
\begin{enumerate}
\item[(i)] $P \equiv 2 \pmod{4}$, \\
\item[(ii)] if $\chi(n) \neq 0$, then $n$ is odd \\
\item[(iii)] $\chi$ is even modulo $P$, i.e., $\chi(P-n) = \chi(n)$.    
\end{enumerate} 
If $q$ is a primitive $N$th root of unity, then 
\begin{equation*}
2\sum_{n=0}^{PN} \left(1-\frac{n}{PN}\right) \chi(n) q^{\frac{n^2}{4P}} =  \sum_{n=0}^{PN} \chi(n) q^{\frac{n^2}{4P}}.
\end{equation*}
\end{lemma}
      
\begin{proof}
We have 
\begin{align}
2\sum_{n=0}^{PN}& \left(1-\frac{n}{PN}\right) \chi(n) q^{\frac{n^2}{4P}} \nonumber \\
&= \sum_{n=0}^{PN} \left(1-\frac{n}{PN}\right) \chi(n) q^{\frac{n^2}{4P}} + \sum_{n=0}^{PN} \left(1-\frac{n}{PN}\right) \chi(n) q^{\frac{n^2}{4P}} \nonumber \\
&= \sum_{n=0}^{PN} \left(1-\frac{n}{PN}\right) \chi(n) q^{\frac{n^2}{4P}}  + \sum_{n=0}^{PN} \left(1-\frac{(PN-n)}{PN}\right) \chi(PN-n) q^{\frac{(PN-n)^2}{4P}} \nonumber \\
&= \sum_{n=0}^{PN} \left(1-\frac{n}{PN}\right) \chi(n) q^{\frac{n^2}{4P}}  + \sum_{n=0}^{PN} \frac{n}{PN} \chi(PN-n) q^{\frac{(P^2N^2 - 2nPN + n^2)}{4P}} \nonumber \\
&= \sum_{n=0}^{PN} \left(1-\frac{n}{PN}\right) \chi(n) q^{\frac{n^2}{4P}}  + \sum_{n=0}^{PN} \frac{n}{PN} q^{\frac{n^2}{4P}} \chi(PN-n) q^{\frac{N(PN - 2n)}{4}}. \label{tbc}
\end{align}
Consider the term
$$
\chi(PN-n) q^{\frac{N(PN - 2n)}{4}}.
$$
We have two cases according to the parity of $N$.    First, if $N$ is odd then since $P \equiv 2 \pmod{4}$ and $n$ is odd we have that $PN-2n$ is a multiple of $4$.    Therefore $q^{\frac{N(PN - 2n)}{4}} = 1$.   Moreover, by $(iii)$ we have $\chi(PN - n) =\chi(n)$.    Hence
$$
\chi(PN-n) q^{\frac{N(PN - 2n)}{4}} = \chi(n).
$$
If $N$ is even, then $PN-2n \equiv 2 \pmod{4}$ and so $q^{\frac{N(PN - 2n)}{4}} = -1$. Now since $N$ is even we have $\chi(PN - n) = -\chi(n)$.   Once again, we find that 
$$
\chi(PN-n) q^{\frac{N(PN - 2n)}{4}} = \chi(n).
$$ 
Using this in \eqref{tbc} we have
\begin{align*}
2\sum_{n=0}^{PN} \left(1-\frac{n}{PN}\right) \chi(n) q^{\frac{n^2}{4P}} &= \sum_{n=0}^{PN} \left(1-\frac{n}{PN}\right) \chi(n) q^{\frac{n^2}{4P}}  + \sum_{n=0}^{PN} \frac{n}{PN} q^{\frac{n^2}{4P}} \chi(n) \\
&= \sum_{n=0}^{PN} \chi(n) q^{\frac{n^2}{4P}},
\end{align*}
as desired.
\end{proof}

\begin{remark}\label{chiproprmk}
Note that (i) and (ii) in Lemma \ref{chiprop} are satisfied for all of the functions $\chi_{\bm{p}}^{\bm{\ell}}$ considered in Theorem \ref{main}. Moreover, $p_1 = 2$ and so by \cite[(3.8)]{hikami3} the involution $\sigma_1$ fixes $\bm{\ell} = (1,1,\ell_3)$, i.e., $\chi_{\bm{p}}^{\bm{\ell}}(n+P) = -\chi_{\bm{p}}^{\bm{\ell}}(n)$. Now replace $n$ with $-n$ and use the fact that $\chi_{\bm{p}}^{\bm{\ell}}$ is odd to obtain (iii) in Lemma \ref{chiprop}. Finally, note that $\chi_{\bm{p}}^{\bm{\ell}}$ is odd modulo $NP$ for all even $N$ and even modulo $NP$ for all odd $N$.    
\end{remark}

Next, we recall a result (see \cite[Proposition 3]{hikami1} or \cite[Corollary 3.9]{tm}) which explicitly computes the limiting values of the Eichler integrals $\widetilde{\Phi}_{\bm{p}}^{\bm{\ell}}(\tau)$ as $\tau \to \frac{M}{N}$.

\begin{proposition} \label{Elimit} Let $M$ and $N$ be coprime positive integers. We have
\begin{equation*}
\lim_{\tau \to \frac{M}{N}} \widetilde{\Phi}_{\bm{p}}^{\bm{\ell}}(\tau) = \sum_{n=0}^{PN} \chi_{\bm{p}}^{\bm{\ell}} \left( 1 - \frac{n}{PN} \right) \zeta_N^{M\frac{n^2}{4P}}.
\end{equation*}
\end{proposition} 

\subsection{Relations between incomplete quadratic Gauss sums}

For the fixed tuples $\bm{p} = (p_1, p_2, p_3)$ of pairwise coprime positive integers and $\bm{\ell} = (\ell_1, \ell_2, \ell_3) \in \mathbb{Z}^3$ satisfying $0 < \ell_j < p_j$, we define $m(\bm{\epsilon})\in\{0,\ldots,2P-1\}$ such that 
\begin{equation} \label{m}
m(\bm{\epsilon})\equiv P\left (1 + \sum_{j=1}^{3} \frac{\epsilon_j \ell_j}{p_j} \right) \pmod{2P}.
\end{equation}
Note that $m(\bm \epsilon)^2 \pmod{2P}$ is independent of $\bm{\epsilon} \in\{\pm 1\}^3$. We
write
\begin{equation*} \label{6ms}
T = T(q) = -\sum_{\bm \epsilon \in\{\pm 1\}^3}
\epsilon_1\epsilon_2\epsilon_3\,
T_{\bm \epsilon},\qquad\quad
S = S(q) =  -\sum_{\bm \epsilon \in\{\pm 1\}^3}
\epsilon_1\epsilon_2\epsilon_3\,
S_{\bm \epsilon}
\end{equation*}
where
\begin{equation*} \label{6ms2}
T_{\bm \epsilon} = T_{\bm \epsilon} (q) \coloneqq
\sum_{\substack{n\equiv m(\bm \epsilon) \bmod 2P\\0\leq n\leq 2p_3 N}} q^{\frac{n^2}{4P}},\qquad\quad
S_{\bm \epsilon} = S_{\bm \epsilon} (q) \coloneqq
\sum_{\substack{n\equiv m(\bm \epsilon) \bmod 2P\\2p_3 N< n\leq 4p_3 N}} q^{\frac{n^2}{4P}}.
\end{equation*}
If we write $n = k2P+m(\bm \epsilon)$ for some positive integer $k$, then
\begin{equation} \label{rewrite}
T_{\bm \epsilon} =
\sum_{k=0}^{\left \lfloor\frac {2p_3 N-m(\bm \epsilon)}{2P} \right \rfloor}
q^{{Pk^2 + m(\bm \epsilon) k + \frac{m(\bm \epsilon)^2}{4P}}},
\qqquad
S_{\bm \epsilon} =
\sum_{k=\left \lceil \frac {2p_3 N-m(\bm \epsilon)}{2P} \right \rceil}^{\left  \lfloor\frac {4p_3 N-m(\bm \epsilon)}{2P} \right \rfloor}
q^{{Pk^2 + m(\bm \epsilon) k + \frac{m(\bm \epsilon)^2}{4P}}}
\end{equation}
and so $T_{\bm \epsilon}$, $S_{\bm \epsilon} \in q^{\frac{m(\bm \epsilon)^2}{4P}} \mathbb{Z}[q]$. The sums (\ref{rewrite}) satisfy the following relations when evaluated at a root of unity.

\begin{proposition} \label{pieces1}
Let $\bm{p} = (2,3, p_3)$ with $(6,p_3)=1$, $\bm{\ell} = (1, 1, \ell_3) \in \mathbb{Z}^3$ satisfy $0 < \ell_3 < p_3$ and $\bm \epsilon := (\epsilon_1, \epsilon_2, \epsilon_3)$, $\bm \epsilon' := (\epsilon_1', \epsilon_2', \epsilon_3') \in \{\pm 1\}^3$. If $M$ and $N$ are coprime positive integers and $q= e^{\frac{2\pi i M}{N}}$, then the following are true:
\begin{enumerate}
\item[(i)] If $\frac{\epsilon_1+\epsilon_1'}{2} \equiv N \pmod 2$,
$
\epsilon_2+\epsilon_2' \equiv N \pmod 3$ and $
\epsilon_3' = -\epsilon_3,
$
then $\frac{m(\bm \epsilon)+m(\bm \epsilon')}{2P} \equiv \frac {N} 6 \pmod 1$ and
$$T_{\bm \epsilon} = e^{\frac{\pi i \left( p_3N-m(\bm \epsilon') \right)M}{3}}  T_{\bm \epsilon'}.$$

\item[(ii)] If $\frac{\epsilon_1-\epsilon_1'}{2} \equiv N \pmod 2$,
$\epsilon_2-\epsilon_2' \equiv N \pmod 3$ and $
\epsilon_3' = \epsilon_3$,
then $\frac{m(\bm \epsilon)-m(\bm \epsilon')}{2P} \equiv \frac {N} 6 \pmod 1$ and 
$$S_{\bm \epsilon} = e^{\frac{\pi i \left( p_3N+m(\bm \epsilon') \right)M}{3}} T_{\bm \epsilon'}.$$

\item[(iii)] If $\frac{\epsilon_1+\epsilon_1'}{2} \equiv 0\pmod 2$,
$\epsilon_2 +\epsilon_2' \equiv 2N \pmod 3$ and $\epsilon_3 = -\epsilon_3'$,
then $\frac{m(\bm \epsilon)+m(\bm \epsilon')}{2P}
\equiv \frac {N} 3 \pmod 1$ and 
$$S_{\bm \epsilon} = e^{\frac{2\pi i \left(2p_3N- m(\bm \epsilon') \right)M}{3}} T_{\bm \epsilon'}.$$

\item[(iv)] If $\frac{\epsilon_1+\epsilon_1'}{2} \equiv N \pmod 2$,
$\epsilon_2 +\epsilon_2'\equiv 0 \pmod 3$
 and $\epsilon_3 = -\epsilon_3'$,
then $\frac{m(\bm \epsilon)+m(\bm \epsilon')}{2P}
\equiv \frac {N} 2 \pmod 1$ and
$$S_{\bm \epsilon} =  e^{\pi i \left( p_3N-m(\bm \epsilon') \right)M} S_{\bm \epsilon'}.$$
\end{enumerate}
\end{proposition}

\begin{proof}
We begin with (i). By (\ref{m}) and simplifying, we have
\begin{equation} \label{summmp}
m(\bm \epsilon) + m(\bm \epsilon') \equiv 3p_3 (\epsilon_1+\epsilon_1') + 2p_3 (\epsilon_2+\epsilon_2') + 6\ell_3 (\epsilon_3+\epsilon_3') \pmod{2P}.
\end{equation}
The conditions on $\bm \epsilon$ and $\bm \epsilon'$ combined with (\ref{summmp}) now imply that 
\begin{equation} \label{sumid1}
\frac{m(\bm \epsilon) + m(\bm \epsilon')}{2P} \equiv \frac {N} 6 \pmod 1.
\end{equation}
Since $m(\bm \epsilon)$, $m(\bm \epsilon') \in [0, 2P)$, it follows that $\frac{m(\bm \epsilon)}{2P}$, $\frac{m(\bm \epsilon')}{2P} \in[0,1)$. By~\eqref{sumid1}, we have
\begin{equation}\label{floorid11}
\frac N6 - \frac{m(\bm \epsilon)}{2P} - \left \lfloor\frac{N}6 - \frac{m(\bm \epsilon)}{2P} \right \rfloor = \frac{m(\bm \epsilon')}{2P}
\end{equation}
and so rearranging terms and applying the floor function yields
\begin{equation}\label{floorid12}
\left \lfloor \frac{N}6 - \frac{m(\bm \epsilon)}{2P} \right \rfloor
=\left \lfloor\frac{N}6 - \frac{m(\bm \epsilon')}{2P} \right \rfloor.
\end{equation}
From (\ref{rewrite}), we compute
\begin{equation} \label{t2t}
\begin{aligned} 
T_{\bm \epsilon} &=\sum_{k=0}^{\left \lfloor \frac N6-\frac {m(\bm \epsilon)}{2P} \right \rfloor}
q^{P k^2+m(\bm \epsilon)k +\frac{m(\bm \epsilon)^2}{4P}}\\ 
&=\sum_{k=0}^{\left \lfloor \frac N6-\frac {m(\bm \epsilon)}{2P} \right \rfloor}
q^{P \left ( \left \lfloor \frac N6-\frac {m(\bm \epsilon)}{2P} \right \rfloor-k \right)^2
+ m(\bm \epsilon) \left (\left \lfloor \frac N6-\frac {m(\bm \epsilon)}{2P} \right \rfloor-k \right )
+ \frac{m(\bm \epsilon)^2}{4P}}
\\ 
&= q^{P \left \lfloor \frac N6-\frac {m(\bm \epsilon)}{2P} \right \rfloor^2
+ m(\bm \epsilon)  \left \lfloor \frac N6-\frac {m(\bm \epsilon)}{2P} \right \rfloor
+\frac{m(\bm \epsilon)^2-m(\bm \epsilon')^2}{4P}} \\ 
&\qqquad\times\sum_{k=0}^{\left \lfloor \frac N6-\frac {m(\bm \epsilon)}{2P} \right \rfloor}
q^{k^2
- \left( 2P \left \lfloor\frac N6-\frac {m(\bm \epsilon)}{2P} \right \rfloor
  + m(\bm \epsilon) \right )k
  +\frac{m(\bm \epsilon')^2}{4P}}
\\
& = q^{\frac{N \left(p_3N-m(\bm \epsilon') \right)}{6}}
\sum_{k=0}^{\left \lfloor\frac N6-\frac {m(\bm \epsilon')}{2P} \right \rfloor}
q^{k^2+m(\bm \epsilon')k +  \frac{PNk}{3} +\frac{m(\bm \epsilon')^2}{4P}}
\end{aligned}
\end{equation}
where we have applied the substitution $k \mapsto \left \lfloor \frac N6-\frac{m(\bm \epsilon)}{2P} \right \rfloor -k$
and used \eqref{floorid11} and \eqref{floorid12}. Letting $q=e^{\frac{2\pi i M}{N}}$ in (\ref{t2t}) and simplifying yields the result.

For (ii), a computation similar to that of \eqref{summmp} and the conditions on $\bm \epsilon$ and $\bm \epsilon'$ imply
\bea\label{diffmm1}
\frac{m(\bm \epsilon)-m(\bm \epsilon')}{2P} \equiv \frac {N} 6 \pmod 1.
\eea
By (\ref{diffmm1}), we have
\begin{equation}\label{floorid20}
\frac N 6 - \frac{m(\bm \epsilon)}{2P}
-\left\lfloor\frac N 6 - \frac{m(\bm \epsilon)}{2P} \right\rfloor 
=-\frac{m(\bm \epsilon')}{2P}+1
\end{equation}
and so again rearranging terms and applying the floor function leads to
\begin{equation}\label{floorid22}
\left\lfloor\frac N 3 - \frac{m(\bm \epsilon)}{2P}\right\rfloor
-\left\lfloor\frac N 6 - \frac{m(\bm \epsilon)}{2P}\right\rfloor 
=\left\lfloor\frac N 6-\frac{m(\bm \epsilon')}{2P}\right\rfloor+1.
\end{equation}
From (\ref{rewrite}), we compute
\begin{equation} \label{s2t}
\begin{aligned}
S_{\epsilon}  &= \sum_{k=\left \lfloor \frac N6-\frac {m(\bm \epsilon)}{2P} \right \rfloor+1}
^{\left \lfloor \frac N3-\frac {m(\bm \epsilon)}{2P} \right \rfloor}
q^{P k^2+m(\bm \epsilon)k +\frac{m(\bm \epsilon)^2}{4P}}\\
&=\sum_{k=0}^{\left \lfloor \frac N 3 - \frac{m(\bm \epsilon)}{2P}\right\rfloor -\left\lfloor\frac N 6 - \frac{m(\bm \epsilon)}{2P}\right\rfloor-1}
q^{P \left ( \left \lfloor \frac{N}6-\frac{m(\bm \epsilon)}{2P} \right\rfloor+1+k \right)^2 + m(\bm \epsilon) \left ( \left\lfloor \frac N6-\frac{m(\bm \epsilon)}{2P} \right \rfloor+1+k \right) +\frac{m(\bm \epsilon)^2}{4P}}\\
&=q^{P \left (\left \lfloor \frac{N}6-\frac{m(\bm \epsilon)}{2P} \right \rfloor+1 \right)^2
+m(\bm \epsilon)\left ( \left \lfloor \frac N6-\frac{m(\bm \epsilon)}{2P} \right \rfloor+1 \right)
+\frac{m(\bm \epsilon)^2}{4P}}
\\
&\qqquad\times\sum_{k=0}^{\left\lfloor\frac N 3 - \frac{m(\bm \epsilon)}{2P}\right\rfloor
-\left\lfloor\frac N 6 - \frac{m(\bm \epsilon)}{2P}\right\rfloor-1}
q^{P k^2+ \left ( m(\bm \epsilon) + 2P \left \lfloor \frac N6-\frac{m(\bm \epsilon)}{2P} \right \rfloor+ 2P \right) k
+\frac{m(\bm \epsilon')^2}{4P}} \\ 
&= q^{\frac{N \left (p_3N+m(\bm \epsilon') \right)}{6}}
\sum_{k=0}^{\left \lfloor \frac N 6  - \frac{m(\bm \epsilon')}{2P} \right \rfloor}
q^{P k^2+m(\bm \epsilon')k + \frac{PNk}{3} + \frac{m(\bm \epsilon')^2}{4P}}
\end{aligned}
\end{equation}
where we have applied the substitution $k \mapsto \left \lfloor \frac{N}6-\frac{m(\bm \epsilon)}{2P} \right \rfloor+1+k$ and used
\eqref{floorid20} and \eqref{floorid22}. Letting $q=e^{\frac{2\pi iM}{N}}$ in (\ref{s2t}) and simplifying yields the result.

For (iii), a computation similar to that of \eqref{summmp} and the conditions on $\bm \epsilon$ and $\bm \epsilon'$ imply
\bea\label{diffmm4}
\frac{m(\bm \epsilon)+m(\bm \epsilon')}{2P} \equiv \frac {N} 3 \pmod 1.
\eea
By (\ref{diffmm4}), we have
\begin{equation} \label{floorid41}
\frac {N} 3 -\frac{m(\bm \epsilon)}{2P}
-\left\lfloor
\frac {N} 3 -\frac{m(\bm \epsilon)}{2P}
\right\rfloor
=
\frac{m(\bm \epsilon')}{2P}
\end{equation}
and so again rearranging terms and applying the floor function leads to
\bea\label{floorid42}
\left\lfloor
\frac {N} 3 -\frac{m(\epsilon)}{2P}
\right\rfloor
-
\left\lfloor
\frac {N} 6 -\frac{m(\epsilon)}{2P}
\right\rfloor-1=
\left\lfloor\frac N 6 -\frac{m(\epsilon')}{2P}\right\rfloor.
\eea
Here, we have used $\lfloor-x\rfloor = -\lfloor x\rfloor -1$ for $x\in\mathbb R\setminus\Z$. From (\ref{rewrite}), we compute
\begin{equation} \label{s2t2}
\begin{aligned}
S_{\bm \epsilon}  &=\sum_{k=\left \lfloor \frac N6-\frac {m(\bm \epsilon)}{2P} \right \rfloor+1}
^{\left \lfloor \frac N3-\frac {m(\bm \epsilon)}{2P} \right \rfloor}
q^{P k^2+m(\bm \epsilon)k +\frac{m(\bm \epsilon)^2}{4P}}\\
&=\sum_{k=0}^{\left\lfloor\frac N 3 - \frac{m(\bm \epsilon)}{2P} \right\rfloor-\left\lfloor\frac N 6 - \frac{m(\bm \epsilon)}{2P}\right\rfloor-1}
q^{P \left ( \left \lfloor \frac{N}{3} - \frac{m(\bm \epsilon)}{2P} \right \rfloor - k \right )^2 + m(\bm \epsilon) \left( \left \lfloor \frac N3 -\frac{m(\bm \epsilon)}{2P} \right \rfloor- k \right) + \frac{m(\bm \epsilon)^2}{4P}}  \\
&=q^{P \left( \left \lfloor \frac{N}3-\frac{m(\bm \epsilon)}{2P} \right \rfloor \right )^2 + m(\bm \epsilon) \left \lfloor \frac N3-\frac{m(\bm \epsilon)}{2P}\right \rfloor  + \frac{m(\bm \epsilon)^2}{4P}} \\
&\qqquad\times\sum_{k=0}^{\left\lfloor\frac N 3 - \frac{m(\bm \epsilon)}{2P} \right\rfloor
-\left\lfloor\frac N 6 - \frac{m(\bm \epsilon)}{2P} \right\rfloor-1}
q^{P k^2- \left ( m(\bm \epsilon) + 2P \left \lfloor \frac N3-\frac{m(\bm \epsilon)}{2P}\right \rfloor \right) k +\frac{m(\bm \epsilon')^2}{4P}}\\
&= q^{\frac{N \left ( PN/3 - m(\bm \epsilon') \right)}{3}}
\sum_{k=0}^{\left \lfloor \frac N 6  - \frac{m(\bm \epsilon')}{2P} \right \rfloor}
q^{P k^2 + m(\bm \epsilon') k + \frac{2PNk}{3} + \frac{m(\bm \epsilon')^2}{4P}}
\end{aligned}
\end{equation}
where we have applied the substitution $k \mapsto \left \lfloor \frac{N}3-\frac{m(\bm \epsilon)}{2P} \right \rfloor-k$ and used \eqref{floorid41} and \eqref{floorid42}. Letting $q=e^{\frac{2\pi i M}{N}}$ in (\ref{s2t2}) and simplifying yields the result.

In order to prove (iv), we will use (ii) and (iii). A computation similar to that of \eqref{summmp} and the conditions on $\bm \epsilon$ and $\bm \epsilon'$ imply
\begin{equation} \label{mod1}
\frac{m(\bm \epsilon)+m(\bm \epsilon')}{2P} \equiv \frac N 2 \pmod 1.
\end{equation}
We now distinguish between two different cases. First, if $2N-\epsilon_2 \not \equiv  0 \pmod 3$, we define $\bm \epsilon'' =(\epsilon_1'', \epsilon_2'', \epsilon_3'') \in\{\pm 1\}^3$ by
 \bea\label{eep1}
 \frac{\epsilon_1 + \epsilon_1''}{2} \equiv 0 \pmod 2,\qqquad
 \epsilon_2+\epsilon_2'' \equiv 2N \pmod 3,\qqquad
 \epsilon_3=-\epsilon_3''.
 \eea
Then (iii) with conditions (\ref{eep1}) and (\ref{mod1}) imply
\begin{equation} \label{st1}
S_{\bm \epsilon} = e^{\frac{2 \pi i \left (2p_3 N - m(\bm \epsilon'')\right)M}{3}} T_{\bm \epsilon''}
 = e^{\frac{2 \pi i \left (-2p_3 N - m(\bm \epsilon') \right)M}{3}} T_{\bm \epsilon''}.
 \end{equation}
Subtracting the equations in (\ref{eep1}) from the initial assumptions on $\bm \epsilon$ and $\bm \epsilon'$ yields
 \bea \label{newc}
 \frac{\epsilon_1'-\epsilon_1''}{2} \equiv N \pmod 2,\qqquad
 \epsilon_2'-\epsilon_2'' \equiv N \pmod 3,\qqquad
 \epsilon_3'=\epsilon_3''.
 \eea
Thus (ii) with conditions (\ref{newc}) imply
\begin{equation} \label{st2}
T_{\bm \epsilon''} = e^{\frac{-\pi i \left ( p_3 N + m(\bm \epsilon'') \right) M}{3}} S_{\bm \epsilon'} = e^{\frac{-\pi i \left ( -p_3 N + m(\bm \epsilon') \right) M}{3}} S_{\bm \epsilon'}.
\end{equation} 
We now combine (\ref{st1}) and (\ref{st2}) and simplify to obtain the result. Second, if $2N-\epsilon_2 \equiv  0 \bmod 3$, then $N-\epsilon_2 \not\equiv 0 \bmod 3$ and we define 
$\bm \epsilon''' = (\epsilon_1''', \epsilon_2''', \epsilon_3''') \in\{\pm 1\}^3$ by
 \bea \label{eep2}
 \frac{\epsilon_1-\epsilon_1'''}{2} \equiv N \pmod 2,\qqquad
 \epsilon_2-\epsilon_2''' \equiv N \pmod 3,\qqquad
 \epsilon_3=\epsilon_3'''.
 \eea
Thus (ii) with conditions (\ref{eep2}) imply 
\begin{equation} \label{st3}
S_{\bm \epsilon} = e^{\frac{\pi i \left( p_3 N + m(\bm \epsilon''') \right) M}{3}} T_{\bm \epsilon'''} = e^{\frac{\pi i \left( -p_3 N - m(\bm \epsilon') \right) M}{3}} T_{\bm \epsilon'''}. 
\end{equation}
Subtracting the equations in (\ref{eep2}) from the initial assumptions on $\bm \epsilon$ and $\bm \epsilon'$, we obtain
 \bea \label{newc2}
 \frac{\epsilon_1'+\epsilon_1'''}{2} \equiv 0 \pmod 2,\qqquad
 \epsilon_2'+\epsilon_2''' \equiv 2N \pmod 3,\qqquad
 \epsilon_3'=-\epsilon_3'''.
 \eea
 Thus (iii) with conditions (\ref{newc2}) imply
 \begin{equation} \label{st4}
 T_{\bm \epsilon'''} = e^{\frac{-2\pi i \left( 2p_3 N - m(\bm \epsilon''') \right) M}{3}} S_{\bm \epsilon'} = e^{\frac{-2\pi i \left( 4p_3 N + m(\bm \epsilon') \right) M}{3}} S_{\bm \epsilon'}.
 \end{equation}
 We now combine (\ref{st3}) and (\ref{st4}) and simplify to obtain the result.
 \end{proof}

Applying the identities from Proposition \ref{pieces1} immediately yields the following result. Henceforth, we use the convention $+$ for $+1$ and $-$ for $-1$ in the indices of $T_{\bm \epsilon}$ and $S_{\bm \epsilon}$. 

\begin{corollary}\label{pieces}
Let $\bm{p} = (2,3, p_3)$ with $(6,p_3)=1$ and $\bm{\ell} = (1, 1, \ell_3) \in \mathbb{Z}^3$ satisfy $0 < \ell_3 < p_3$. If $M$ and $N$ are coprime positive integers and $q=e^{\frac{2\pi i M}{N}}$, then the following are true:
 \begin{enumerate}
   \item If $N\equiv 0 \bmod 6$, then
  \begin{gather*}\label{N0}
  T_{+++}
    =e^{-\frac{\pi ip_3 M}{3}} T_{---}
    =e^{\frac{\pi ip_3 M}{3}} S_{+++}
    =e^{-\frac{2\pi ip_3 M}{3}}S_{---},\\
  T_{++-}
    =e^{-\frac{\pi ip_3 M}{3}} T_{--+}
    =e^{-\frac{2\pi ip_3 M}{3}} S_{++-}
    =e^{\frac{\pi ip_3 M}{3}} S_{--+},\\
  T_{+-+}
    =e^{\frac{\pi ip_3 M}{3}} T_{-+-}
    =e^{-\frac{\pi ip_3 M}{3}} S_{+-+}
    =e^{\frac{2\pi ip_3 M}{3}} S_{-+-},\\
  T_{+--}
    =e^{\frac{\pi ip_3 M}{3}} T_{-++}
    =e^{-\frac{\pi ip_3 M}{3}} S_{+--}
    =e^{\frac{2\pi ip_3 M}{3}} S_{-++}.
  \end{gather*}
  \item If $N\equiv 1 \bmod 6$, then 
  \begin{gather*}
   T_{+-+} = T_{+--},\qqquad
   T_{--+} = T_{---}
   ,\\
   T_{-++} = S_{+-+} = S_{++-},\qqquad
   T_{-+-} = S_{+--} = S_{+++}
   ,\\
   T_{+++} = S_{--+} = S_{-+-},\qqquad
   T_{++-} = S_{---} = S_{-++}.
\end{gather*}
  \item If $N\equiv 2 \bmod 6$, then
   \begin{gather*}
   T_{+++} = -T_{-+-},\qqquad
   T_{++-} = -T_{-++},\\
   T_{+-+} = -S_{+++}=S_{---},\qqquad
   T_{+--} = -S_{++-}=S_{--+},\\
   T_{--+} = -S_{-++}=S_{+--},\qqquad
   T_{---} = -S_{-+-}=S_{+-+}.
  \end{gather*}
  \item If $N\equiv 3 \bmod 6$, then
  \begin{gather*}
  T_{+++}
    =e^{\frac{2\pi M}{3}} T_{+--}
    =e^{-\frac{2\pi i M}{3}}S_{-++}
    =e^{-\frac{2\pi i M}{3}}S_{---},\\
  T_{++-}
    =e^{\frac{2\pi i M}{3}}T_{+-+}
    =e^{-\frac{2\pi i M}{3}}S_{-+-}
    =e^{-\frac{2\pi i M}{3}}S_{--+},\\
  T_{-++}
    =e^{\frac{2\pi i M}{3}}T_{---}
    =e^{-\frac{2\pi i M}{3}}S_{+++}
    =e^{-\frac{2\pi i M}{3}}S_{+--},\\
  T_{-+-}
    =e^{\frac{2\pi i M}{3}}T_{--+}
    =e^{-\frac{2\pi i M}{3}} - S_{++-}
    =e^{-\frac{2\pi i M}{3}} - S_{+-+}.
  \end{gather*}
  \item If $N\equiv 4 \bmod 6$, then
   \begin{gather*}
   T_{+-+} = -T_{---},\qqquad
   T_{+--} = -T_{--+},\\
   T_{+++} = -S_{+-+}=S_{-+-},\qqquad
   T_{++-} = -S_{+--}=S_{-++},\\
   T_{-++} = -S_{--+}=S_{++-},\qqquad
   T_{-+-} = -S_{---}=S_{+++}.
  \end{gather*}
  \item If $N\equiv 5 \bmod 6$, then 
  \begin{gather*}
   T_{+++} = T_{++-},\qqquad
   T_{-++} = T_{-+-},\\
   T_{--+} = S_{+++}=S_{+--},\qqquad
   T_{---} = S_{++-}=S_{+-+},\\
   T_{+-+} = S_{-++}=S_{---},\qqquad
   T_{+--} = S_{-+-}=S_{--+}.
\end{gather*}
 \end{enumerate}
\end{corollary}

We now illustrate the previous result.

\begin{corollary} \label{ex} Let $\bm{p} = (2,3, p_3)$ with $(6,p_3)=1$ and $\bm{\ell} = (1, 1, \ell_3) \in \mathbb{Z}^3$ satisfy $0 < \ell_3 < p_3$ and write $\chi = \chi_{(2, 3, p_3)}^{(1,1,\ell_3)}$. If $M$ and $N$ are coprime positive integers and $q=e^{\frac{2\pi i M}{N}}$, then
\bea\label{gensum4sum}
\sum_{n=0}^{6p_3N} \chi(n) q^{\frac{n^2}{24p_3}} = 4 \sum_{n=0}^{2p_3N} \chi(n) q^{\frac{n^2}{24p_3}}.
\eea
\end{corollary}
\begin{proof}
By Remark~\ref{chiproprmk}, we obtain
\begin{align*}
\sum_{n=0}^{6p_3N} \chi(n) q^{\frac{n^2}{24p_3}} &\= 
-2\sum_{\bm \epsilon\in\{\pm1\}^3} \epsilon_1\epsilon_2\epsilon_3 T_{\bm \epsilon}
- \sum_{\bm \epsilon \in \{\pm1\}^3}
\epsilon_1\epsilon_2\epsilon_3  S_{\bm \epsilon} \\ \nonumber
&\=2 \left( - T_{ + + + } + T_{ + + - } + T_{ + - + } - T_{ + - - } + T_{ - + + } - T_{ - + - } - T_{ - - + } + T_{ - - - } \right) \\ \nonumber
&\qquad - S_{ + + + } + S_{ + + - } + S_{ + - + } - S_{ + - - } + S_{ - + + } - S_{ - + - } - S_{ - - + } + S_{ - - - }
\end{align*}
and
\begin{align*}
\sum_{n=0}^{2p_3N} \chi(n) q^{\frac{n^2}{24p_3}} &\= -\sum_{\bm \epsilon\in\{\pm1\}^3} \epsilon_1\epsilon_2\epsilon_3 T_{\bm \epsilon} \\ \nonumber
&\=
- T_{ + + + } + T_{ + + - } + T_{ + - + } - T_{ + - - } + T_{ - + + } - T_{ - + - } - T_{ - - + } + T_{ - - - }.
\end{align*}
Hence, to verify \eqref{gensum4sum} it suffices to prove that
\bea\label{eq:TsSs}
0 &\= -2 \left(- T_{ + + + } + T_{ + + - } + T_{ + - + } - T_{ + - - } + T_{ - + + } - T_{ - + - } - T_{ - - + } + T_{ - - - } \right) \\
&\qquad - S_{ + + + } + S_{ + + - } + S_{ + - + } - S_{ + - - } + S_{ - + + } - S_{ - + - } - S_{ - - + } + S_{ - - - }.
\eea
If $N \not \equiv 0 \pmod{3}$, this follows directly from Corollary \ref{pieces} (2), (3), (5) and (6). Otherwise, the terms on the right-hand side of \eqref{eq:TsSs} have to be rearranged. For example, if $N\equiv 0 \pmod 6$, by Corollary \ref{pieces} (1), the right-hand side of \eqref{eq:TsSs} is given by
\begin{equation} \label{Nmod06}
 \begin{aligned}
 &T_{+++} \left( 2-2e^{\frac{\pi i p_3 M}{3}} - e^{-\frac{\pi i p_3 M}{3}} + e^{\frac{2\pi i p_3 M}{3}} \right) +T_{++-} \left(-2+2e^{\frac{\pi i p_3 M}{3}} + e^{-\frac{\pi i p_3 M}{3}} - e^{\frac{2\pi i p_3 M}{3}} \right) \\
&+T_{+-+} \left(-2+2e^{-\frac{\pi i p_3 M}{3}} + e^{\frac{\pi i p_3 M}{3}} - e^{-\frac{2\pi i p_3 M}{3}} \right) + T_{+--} \left( 2-2e^{-\frac{\pi i p_3 M}{3}} - e^{\frac{\pi i p_3 M}{3}} + e^{-\frac{2\pi i p_3 M}{3}} \right)
 \end{aligned}
 \end{equation}
where $M \not \equiv 0 \pmod{3}$. We then use the identity $2-2e^{\frac{\pi i p_3 M}{3}} - e^{-\frac{\pi i p_3 M}{3}} + e^{\frac{2\pi i p_3 M}{3}} =0$ to deduce \eqref{eq:TsSs} from (\ref{Nmod06}). The proof for $N\equiv 3 \pmod 6$ is similar.
\end{proof}

\section{Proof of Theorem \ref{main}}

We are now in a position to prove Theorem \ref{main}. Again, we follow the order in which the Bailey pairs were introduced in Section \ref{arseBP}.

\begin{proof} We begin with the proof of (\ref{main6}). By Proposition \ref{Elimit}, (\ref{main6}) is equivalent to
\begin{equation} \label{6step0}
-\frac{1}{2} \sum_{n=0}^{6(6p+1)N} \chi(n) \left( 1 - \frac{n}{6(6p+1)N} \right) \zeta_N^{M \frac{n^2 - (6p-1)^2}{24(6p+1)}} = H_{p}^{(6)}(\zeta_N^{M}) 
\end{equation}
where $\chi = \chi_{(2,3,6p+1)}^{(1,1,2p)}$ is the odd periodic function given by 
\begin{eqnarray} \label{chi23}
\chi(n) = 
\begin{cases}
1 & \text{if $n \equiv 42p+5$, $54p+7$, $54p+11$, $66p+13$ $\pmod{72p+12}$,} \\
-1 & \text{if $n \equiv 6p-1$, $18p+1$, $18p+5$, $30p+7$ $\pmod{72p+12}$,}\\
0 & \text{otherwise.}
\end{cases}
\end{eqnarray}
By Theorem \ref{Baileyfamily1},
\begin{equation} \label{6step1}
H_{p}^{(6)}(q) = \sum_{k=0}^{N-1} (-1)^k q^{\binom{k+1}{2} + (p-1)(k^2 + k)} \alpha_k
\end{equation}
where $q=\zeta_N^{M}$ and $(\alpha_k, \beta_k)$ is the Bailey pair in (\ref{BP1}). Next, we claim that
\begin{equation} \label{6step2}
4 \sum_{k=0}^{N-1} (-1)^k q^{\binom{k+1}{2} + (p-1)(k^2 + k)} \alpha_k = - \sum_{n=0}^{6(6p+1)N} \chi(n) q^{\frac{n^2 - (6p-1)^2}{24(6p+1)}}
\end{equation}
when $q=\zeta_N^{M}$. To deduce (\ref{6step2}), we first demonstrate that (as polynomials in $q$)
\begin{equation} \label{6step3}
\sum_{k=0}^{N-1} (-1)^k q^{\binom{k+1}{2} + (p-1)(k^2 + k)} \alpha_k = - \sum_{n=0}^{2(6p+1)N} \chi(n) q^{\frac{n^2 - (6p-1)^2}{24(6p+1)}}.
\end{equation}
We proceed by induction on $N$. The induction step consists of verifying the identity
\begin{equation} \label{6step4}
(-1)^N q^{\binom{N+1}{2} + (p-1)(N^2 + N)} \alpha_N = -\sum_{n=0}^{12p+1} \chi((12p+2)(N+1) - n) q^{\frac{((12p+2)(N+1) - n)^2 - (6p-1)^2}{24(6p+1)}}.
\end{equation}
We consider three cases.  First suppose that $N=3r$. If $r$ is even, then 
\begin{equation*} \label{6step5}
\chi((12p+2)(N+1) - n) = \chi(12p+2-n)
\end{equation*}
and so the right-hand side of \eqref{6step4} is
\begin{equation} \label{6stepeven}
\begin{aligned}
-\sum_{n=0}^{12p+1} \chi(12p+2-n) q^{\frac{((12p+2)(N+1) - n)^2 - (6p-1)^2}{24(6p+1)}} &= 
-\chi(6p-1) q^{\frac{((12p+2)(N+1) - (6p+3))^2 - (6p-1)^2}{24(6p+1)}} \\
&= q^{\frac{((12p+2)(N+1) - (6p+3))^2 - (6p-1)^2}{24(6p+1)}}
\end{aligned}
\end{equation}
by (\ref{chi23}). Similarly, if $r$ is odd, then we have  
\begin{equation*} \label{6step6}
\chi((12p+2)(N+1) - n)  =  \chi((12p+2)(N+3) - 24p-4-n) = -\chi(n+24p+4),
\end{equation*}
and so the right-hand side of \eqref{6step4} is
\begin{equation} \label{6stepodd}
\begin{aligned}
\sum_{n=0}^{12p+1} \chi(n+24p+4) q^{\frac{((12p+2)(N+1) - n)^2 - (6p-1)^2}{24(6p+1)}} &= 
\chi(30p+7) q^{\frac{((12p+2)(N+1) - (6p+3))^2 - (6p-1)^2}{24(6p+1)}} \\
&= -q^{\frac{((12p+2)(N+1) - (6p+3))^2 - (6p-1)^2}{24(6p+1)}}
\end{aligned}
\end{equation}
by (\ref{chi23}). In both (\ref{6stepeven}) and (\ref{6stepodd}), one can check using (\ref{BP1}) that the signs and powers of $q$ on both sides of (\ref{6step4}) match. Next suppose that $N=3r+1$.    If $r$ is even, then 
\begin{equation*} 
\chi((12p+2)(N+1) - n) = \chi(24p+4-n)
\end{equation*}
and so the right-hand side of \eqref{6step4} is
\begin{equation} \label{6stepeven1}
\begin{aligned}
-\sum_{n=0}^{12p+1}& \chi(24p+4-n) q^{\frac{((12p+2)(N+1) - n)^2 - (6p-1)^2}{24(6p+1)}} \\
&= 
-\chi(18p+1) q^{\frac{((12p+2)(N+1) - (6p+3))^2 - (6p-1)^2}{24(6p+1)}} - \chi(18p+5) q^{\frac{((12p+2)(N+1) - (6p-1))^2 - (6p-1)^2}{24(6p+1)}} \\
&= q^{\frac{((12p+2)(N+1) - (6p+3))^2 - (6p-1)^2}{24(6p+1)}} + q^{\frac{((12p+2)(N+1) - (6p-1))^2 - (6p-1)^2}{24(6p+1)}}
\end{aligned}
\end{equation}
by (\ref{chi23}). Similarly, if $r$ is odd, then we have
\begin{equation*} 
\chi((12p+2)(N+1) - n)  =  \chi((12p+2)(N+2) - 12p-2-n) = -\chi(n+12p+2),
\end{equation*}
and so the right-hand side of \eqref{6step4} is
\begin{equation} \label{6stepodd1}
\begin{aligned}
\sum_{n=0}^{12p+1}& \chi(n+12p+2) q^{\frac{((12p+2)(N+1) - n)^2 - (6p-1)^2}{24(6p+1)}} \\
&= \chi(18p+5) q^{\frac{((12p+2)(N+1) - (6p+3))^2 - (6p-1)^2}{24(6p+1)}} + \chi(18p+1) q^{\frac{((12p+2)(N+1) - (6p-1))^2 - (6p-1)^2}{24(6p+1)}} \\
&= -q^{\frac{((12p+2)(N+1) - (6p+3))^2 - (6p-1)^2}{24(6p+1)}} - q^{\frac{((12p+2)(N+1) - (6p-1))^2 - (6p-1)^2}{24(6p+1)}}
\end{aligned}
\end{equation}
by (\ref{chi23}). In both (\ref{6stepeven1}) and (\ref{6stepodd1}), one can check using (\ref{BP1}) that the signs and powers of $q$ on both sides of (\ref{6step4}) match.  As the argument is similar for $N=3r-1$, we omit the details. 

Now, by (\ref{6step3}), it suffices to prove
\begin{equation} \label{6step7}
\sum_{n=0}^{6(6p+1)N} \chi(n) q^{\frac{n^2 - (6p-1)^2}{24(6p+1)}} = 4 \sum_{n=0}^{2(6p+1)N} \chi(n) q^{\frac{n^2 - (6p-1)^2}{24(6p+1)}}
\end{equation}
when $q =\zeta_N^{M}$. This follows by taking $p_3=6p+1$ and multiplying both sides of (\ref{gensum4sum}) by $q^{-\frac{(6p-1)^2}{24(6p+1)}}$ in Corollary \ref{ex}. Thus, (\ref{6step3}) with $q=\zeta_N^{M}$ and (\ref{6step7}) imply (\ref{6step2}). By Lemma \ref{chiprop}, we have
\begin{equation} \label{6step8}
\sum_{n=0}^{6(6p+1)N} \chi(n) q^{\frac{n^2 - (6p-1)^2}{24(6p+1)}} = 2 q^{-\frac{(6p-1)^2}{24(6p+1)}} \sum_{n=0}^{6(6p+1)N} \chi(n) \left( 1 - \frac{n}{6(6p+1)N} \right) q^{\frac{n^2}{24(6p+1)}}
\end{equation}
where $q=\zeta_N^{M}$. Combining (\ref{6step1}), (\ref{6step2}) and (\ref{6step8}), we arrive at
\begin{equation} \label{6step9}
2 \sum_{n=0}^{6(6p+1)N} \chi(n) \left( 1 - \frac{n}{6(6p+1)N} \right) \zeta_N^{M \frac{n^2 - (6p-1)^2}{24(6p+1)}}  = -4 H_{p}^{(6)}(\zeta_N^{M}). 
\end{equation}
Multiplying both sides of (\ref{6step9}) by $-\frac{1}{4}$ yields (\ref{6step0}). Thus, (\ref{main6}) follows.

As the proofs for (\ref{main1eq})--(\ref{main5eq}) and (\ref{main7eq})--(\ref{main10eq}) are similar to that of (\ref{main6}), we sketch the details. One first uses Proposition \ref{Elimit} to express the limiting value of the relevant Eichler integral as a finite sum (as in the left-hand side of (\ref{6step0})). Next, by Theorems \ref{Baileyfamily2}--\ref{Baileyfamily10}, we can express each of the left-hand sides of (\ref{main1eq})--(\ref{main5eq}) and (\ref{main7eq})--(\ref{main10eq}) as a finite sum involving $\alpha_k$, $(q^{p(k^2 + 2k)} \alpha_k)^{*}$ or $\widehat{\alpha}_{k}$ (analogous to (\ref{6step1})). One then proves an underlying identity between polynomials in $q$ akin to (\ref{6step3}). Here we list these identities as significant care is required in some cases. As for (\ref{6step3}), the proofs proceed by induction on $N$. For (\ref{main7eq}), the identity is
\begin{equation} \label{main7poly}
\sum_{k=0}^{N-1} (-1)^k q^{\binom{k+1}{2} + (p-1)(k^2 + k)} \alpha_k = -\sum_{n=0}^{2(6p+1)N} \chi(n) q^{\frac{n^2 - (6p+5)^2}{24(6p+1)}}
\end{equation}
where $(\alpha_k, \beta_k)$ is the Bailey pair (\ref{BP2}) and $\chi=\chi_{(2,3,6p+1)}^{(1,1,2p+1)}$. For (\ref{main3eq}), it is
\begin{equation} \label{main3poly}
\sum_{k=0}^{N-1} (-1)^k q^{\binom{k+1}{2} + (p-1)(k^2 + k)} \alpha_k = -\sum_{n=0}^{2(6p-1)N} \chi(n) q^{\frac{n^2 - (6p-5)^2}{24(6p-1)}}
\end{equation}
where $(\alpha_k, \beta_k)$ is the Bailey pair (\ref{BP3}) and $\chi=\chi_{(2,3,6p-1)}^{(1,1,2p-1)}$. For (\ref{main4eq}), it is
\begin{equation} \label{main4poly}
\sum_{k=0}^{N-1} (-1)^k q^{\binom{k+1}{2} + (p-1)(k^2 + k)} \alpha_k = -\sum_{n=0}^{2(6p-1)N} \chi(n) q^{\frac{n^2 - (6p+1)^2}{24(6p-1)}}
\end{equation}
where $(\alpha_k, \beta_k)$ is the Bailey pair (\ref{BP4}) and $\chi=\chi_{(2,3,6p-1)}^{(1,1,2p)}$. For (\ref{main8eq}), it is
\begin{equation} \label{Hp}
(1-q) \sum_{k=0}^{N-1} (-1)^k q^{\binom{k+1}{2} + (p-1)(k^2 +k)} \alpha_k = \sum_{n=0}^{2(6p+1)N} \chi(n) q^{\frac{n^2 - (6p - 5)^2}{24(6p+1)}}
\end{equation}
where $(\alpha_k, \beta_k)$ is the Bailey pair (\ref{BP5}) and $\chi=\chi_{(2,3,6p+1)}^{(1,1,1)}$. For (\ref{main1eq}), it is
\begin{equation} \label{H1p}
q^{-1}(1-q) \sum_{k=0}^{N-1} (-1)^k q^{\binom{k+1}{2} + (p-1)(k^2 + k)} \alpha_k = \sum_{n=0}^{2(6p-1)N} \chi(n) q^{\frac{n^2 - (6p+5)^2}{24(6p-1)}}
\end{equation}
for $p \geq 2$ and
\begin{equation} \label{H11}
q^{-1}(1-q) \sum_{k=0}^{N-1} (-1)^k q^{\binom{k+1}{2}} \alpha_k = \sum_{n=0}^{10N} \chi(n) q^{\frac{n^2 - 11^2}{120}} - q^{-1} \lambda_N
\end{equation}
where
$$
\lambda_N = \lambda_N(q) := -2 +
\begin{cases}
 (-1)^{r} q^{\frac{15r^2 + r}{2}} +  (-1)^{r} q^{\frac{15r^2 - r}{2}} &\text{if $N=3r$,}\\
 (-1)^{r} q^{\frac{15r^2 + 11r + 2}{2}} &\text{if $N=3r+1$,}\\
(-1)^{r+1} q^{\frac{15r^2 + 19r + 6}{2}}  &\text{if $N=3r+2$}\\
\end{cases}
$$
for $p=1$. Here, $(\alpha_k, \beta_k)$ is the Bailey pair (\ref{BP6}) and $\chi=\chi_{(2,3,6p-1)}^{(1,1,1)}$. One can check that if $q=\zeta_N^{M}$, then $\lambda_N = \lambda_N(\zeta_N^{M}) = -1$. For (\ref{main5eq}), it is
\begin{equation} \label{H5p}
\sum_{k=0}^{N-1} (-1)^k q^{-\binom{k+1}{2}} (q^{p(k^2 + 2k)} \alpha_k)^{*}  = - \sum_{n=0}^{2(6p-1)N} \chi(n) q^{\frac{n^2 - (12p-5)^2}{24(6p-1)}}
\end{equation}
where $(\alpha_k, \beta_k)$ is the Bailey pair (\ref{BP7}) and $\chi=\chi_{(2,3,6p-1)}^{(1,1,3p-1)}$. For (\ref{main9eq}), it is
\begin{equation} \label{main9poly}
\sum_{k=0}^{N-1} (-1)^k q^{-\binom{k+1}{2}} (q^{p(k^2+2k)} \alpha_k)^*
= -\sum_{n=0}^{2(6p+1)N} \chi(n) q^{\frac{n^2 - (12p-1)^2}{24(6p+1)}}
\end{equation}
where $(\alpha_k, \beta_k)$ is the Bailey pair (\ref{BP8}) and $\chi=\chi_{(2,3,6p+1)}^{(1,1,3p)}$. For (\ref{main2eq}), it is
\begin{equation} \label{main2poly}
\begin{aligned}
\sum_{k=0}^{N-1} (-1)^{k} q^{-\binom{k+1}{2} + p(k+1)^2} \widehat{\alpha}_k
 + 1
 & + \frac{1}{2} (-1)^N q^{-\binom{N}{2}+ pN^2 }
 (\widehat{\alpha}_{N-1}-\gamma_{N-1}) \\
 & \qquad \qquad \qquad \qquad \qquad = -\sum_{n=0}^{2(6p-1)N}\chi(n) q^{\frac{n^2-1}{24(6p-1)}}
\end{aligned}
\end{equation}
where $(\alpha_k, \beta_k)$ is the Bailey pair (\ref{BP9}), 
$$
\gamma_k = \gamma_k(q) :=
\begin{cases}
       q^{3r^2-2r}+q^{3r^2-r}   &\text{if $k=3r-1$},\\
     -q^{3r^2+r}               &\text{if $k=3r$},\\
    -q^{3r^2+2r}              &\text{if $k=3r+1$}
\end{cases}
$$
and $\chi=\chi_{(2,3,6p-1)}^{(1,1,p)}$. Again, one can confirm that if $q=\zeta_N^{M}$, then $\gamma_{N-1} =\gamma_{N-1}(\zeta_N^M)= -1$ and $(-1)^N q^{-\binom{N}{2} + pN^2} = -1$. Finally, for (\ref{main10eq}), it is
\begin{equation} \label{main10poly}
\begin{aligned}
  \sum_{k=0}^{N-2}(-1)^k
  q^{-\binom{k+1}{2}+p(k+1)^2}\widehat{\alpha}_k
    -\frac 1 2(-1)^Nq^{-\binom N2+pN^2}\widehat{\alpha}_{N-1} -1
    -\frac{1}{2} (-1)^N q^{pN^2} \kappa_N \\ 
  =-\sum_{n=0}^{2(6p+1)N} \chi(n) q^{\frac{n^2-1}{24(6p+1)}}
 \end{aligned}\end{equation}
 where $(\alpha_k, \beta_k)$ is the Bailey pair (\ref{BP10}),
\begin{align*}
  \kappa_N = \kappa_N(q) := \begin{cases}
   -q^{\frac{r(3r-1)}{2}}-q^{\frac{r(3r+1)}{2}} &\text{if $N = 3r$,}\\
   q^{\frac{r(3r+1)}{2}} &\text{if $N= 3r+1$,}\\
   q^{\frac{(r+1)(3r+2)}{2}} &\text{if $N =3r+2$}\\
    \end{cases}
\end{align*}
and $\chi=\chi_{(2,3,6p+1)}^{(1,1,p)}$. In this case, if $q=\zeta_N^{M}$, then $\kappa_N = \kappa_N(\zeta_N^M) = (-1)^{N+1}$. We now take $q=\zeta_N^{M}$ in each of (\ref{main7poly})--(\ref{main10poly}), simplify and then combine with statements similar to (\ref{6step7}) (resulting from Corollary \ref{pieces}) to establish versions of (\ref{6step2}). After applying Lemma \ref{chiprop} (to obtain variants of (\ref{6step8})), the results follow by combining steps and multiplying by an appropriate factor. In total, this yields (\ref{main1eq})--(\ref{main5eq}) and (\ref{main7eq})--(\ref{main10eq}). 
\end{proof}

\section*{Acknowledgements}
The second and third authors were partially funded by the Irish Research Council Advanced Laureate Award IRCLA/2023/1934. The second author would like to thank the Okinawa Institute of Science and Technology for their hospitality and support during his visit from August 28 to September 15, 2023 as part of the TSVP Thematic Program ``Exact Asymptotics: From Fluid Dynamics to Quantum Geometry". It was during this visit where Professor Hikami kindly shared conjectures (\ref{h6pconj}) and (\ref{h7pconj}). Finally, the third author would like to thank the Max-Planck-Institut f{\"u}r Mathematik for their hospitality and support during the completion of this paper.

\end{document}